\documentclass[12pt]{article}%
\usepackage{amsmath}
\usepackage{amsfonts}
\usepackage{amssymb}
\usepackage{graphicx}%
\providecommand{\U}[1]{\protect\rule{.1in}{.1in}}
\newtheorem {theorem}{Theorem}

\newtheorem {claim}[theorem]{Claim}
\newtheorem {conclusion}[theorem]{Conclusion}

\newtheorem {definition}[theorem]{Definition}

\newtheorem {lemma}[theorem]{Lemma}

\newtheorem {problem}[theorem]{Problem}
\newtheorem {proposition}[theorem]{Proposition}
\newtheorem {remark}[theorem]{Remark}

\newenvironment {proof}[1][Proof]{\noindent \textbf {#1.} }{\ \rule {0.5em}{0.5em}}

\newtheorem {rmk}{Remark}

\newtheorem {prf}{Proof}

\begin{document}

\title{\textbf{An analysis of the superiorization method via the principle of
concentration of measure}}
\author{Yair Censor\thanks{Corresponding author.}\\Department of Mathematics\\University of Haifa\\Mt.\ Carmel, Haifa 3498838, Israel\\(yair@math.haifa.ac.il)
\and Eliahu Levy\\Department of Mathematics\\Technion -- Israel Institute of Technology\\Technion City, Haifa 3200003, Israel\\(eliahu@math.technion.ac.il)}
\date{November 22, 2018. Revised: June 15, 2019.}
\maketitle

\begin{abstract}
The superiorization methodology is intended to work with input data of
constrained minimization problems, i.e., a target function and a constraints
set. However, it is based on an antipodal way of thinking to the thinking that
leads constrained minimization methods. Instead of adapting unconstrained
minimization algorithms to handling constraints, it adapts feasibility-seeking
algorithms to reduce (not necessarily minimize) target function values. This
is done while retaining the feasibility-seeking nature of the algorithm and
without paying a high computational price. A guarantee that the local target
function reduction steps properly accumulate to a global target function value
reduction is still missing in spite of an ever-growing body of publications
that supply evidence of the success of the superiorization method in various
problems. We propose an analysis based on the principle of concentration of
measure that attempts to alleviate the guarantee question of the
superiorization method.\medskip

\end{abstract}


\section{Introduction}

\textbf{The superiorization method studied in this paper. }Let ${{\mathcal{H}%
}}$ be a $J$-dimensional Hilbert space, i.e.,\ the Euclidean space $E^{J}$
with norm $\Vert\cdot\Vert$ and inner product $\left\langle \cdot
,\cdot\right\rangle $, and consider the convex feasibility problem (CFP) which
is to find a point in the nonempty intersection $C$ of a finite number
$C_{1},C_{2},\ldots,C_{I}$ of closed convex sets in ${{\mathcal{H}}}$. Let
$(A_{t})_{t=1}^{\infty}$ be a sequence of operators $A_{t}:{{\mathcal{H}}%
}\rightarrow{{\mathcal{H}}}$ that gives rise to an iterative process which,
starting from an initial $x_{0}\in{{\mathcal{H}}},$ generates a sequence
$(x_{n})_{n}\subset{{\mathcal{H}}}$ by%
\begin{equation}
x_{n+1}:=A_{n+1}(x_{n}),\qquad n=0,1,2,\dots.\label{eq:basic-alg}%
\end{equation}
Further, assume that any sequence $(x_{n})_{n},$ generated by this process
converges, for any initial $x_{0}\in{{\mathcal{H}}},$ to some point
$x_{\infty}\in C.$ An algorithm\footnote{As common, we use the terms algorithm
or algorithmic structure for the iterative processes studied here although no
termination criteria are present and only the asymptotic behavior of these
processes is studied.} that employs such a process is called a
`feasibility-seeking algorithm'\ and will be, henceforth, referred to as a
`basic algorithm'.

Now, consider an iterative process that uses the same algorithmic operators
$(A_{t})_{t=1}^{\infty}$ but perturbs the iterates and generates another
sequence $(x_{n}^{\prime})_{n}\subset{{\mathcal{H}}}$ by%
\begin{equation}
x_{0}^{\prime}=x_{0},\text{ \ \ }x_{n+1}^{\prime}:=A_{n+1}(x_{n}^{\prime
}+\beta_{n}v_{n}),\qquad n=0,1,2,\dots, \label{eq:perturbed-alg}%
\end{equation}
where $v_{n}\in{{\mathcal{H}}}$ and $\beta_{n}$ are real numbers so that
$\Vert v_{n}\Vert\leq M$, are bounded by some $M$, and $\beta_{n}\geq0,$ for
all $n\geq0,$ and $\sum_{n=0}^{\infty}\beta_{n}<+\infty$. Assume that any
sequence $(x_{n}^{\prime})_{n},$ generated by this process, converges to some
point $x_{\infty}^{\prime}\in C.$ An algorithm that employs such a process is
called a `superiorized version of the basic algorithm'. Modifications of this
superiorized version of the basic algorithm have been developed, see, e.g.,
the Appendix, entitled: \textquotedblleft The algorithmic evolution of
superiorization\textquotedblright\ in \cite{linsup-ip17}, however, our current
investigation focuses solely on the above formulation.

The superiorization method (SM) considered here looks at basic algorithms of
the form (\ref{eq:basic-alg}) that are resilient to perturbations as those
that appear in (\ref{eq:perturbed-alg}) and aims at using inexpensive such
perturbations in order to reach (i.e., asymptotically converge to) a feasible
point in $C$ that is superior with respect to some given target function.
These notions are made precise in the next sections.

\textbf{Readings. }To a novice on the SM and perturbation resilience of
algorithms we recommend to read first the recent reviews in
\cite{constanta-weak-strong,herman-rev-14,med-phys-2012}. Current work on
superiorization can be appreciated from the continuously updated Internet page
\cite{sup-bib-online}. For a recent description of previous work that is
related to superiorization but is not included in \cite{sup-bib-online}, such
as the works of Sidky and Pan, e.g., \cite{sidky-pan}, we direct the reader to
\cite[Section 3]{compare2014}. The SNARK14 software package \cite{snark}, with
its in-built capability to superiorize iterative algorithms to improve their
performance, can be helpful to practitioners. Naturally, there is variability
among the bibliography items of \cite{sup-bib-online} in their degree of
relevance to the superiorization methodology and perturbation resilience of
algorithms. In some, superiorization does not appear in the title, abstract or
introduction but only inside the work, e.g., \cite[Subsection 6.2.1:
Optimization vs. Superiorization]{zhang}.

\textbf{A word about the history. }The terms and notions \textquotedblleft
superiorization\textquotedblright\ and \textquotedblleft perturbation
resilience\textquotedblright,\ in the present context, first appeared in the
2009 paper of Davidi, Herman and Censor \cite{dhc-itor-09} which followed its
2007 forerunner by Butnariu, Davidi, Herman and Kazantsev \cite{butnariu}. The
ideas have some of their roots in the 2006 and 2008 papers of Butnariu, Reich
and Zaslavski \cite{brz06,brz08}. All these culminated in Ran Davidi's 2010
PhD dissertation \cite{davidi-phd} and the many papers since then cited in
\cite{sup-bib-online}.

\textbf{The guarantee problem of the SM. }The SM interlaces into a
feasibility-seeking basic algorithm target function reduction steps. These
steps cause the target function to reach lower values locally, prior to
performing the next feasibility-seeking iterations. A mathematical guarantee
has not been found to date that the overall process of the superiorized
version of the basic algorithm will not only retain its feasibility-seeking
nature but also preserve globally the target function reductions. We call this
fundamental question of the SM \textquotedblleft the guarantee problem of the
SM\textquotedblright\ which is: \textquotedblleft under which conditions one
can guarantee that a superiorized version of a bounded perturbation resilient
feasibility-seeking algorithm converges to a feasible point that has target
function value smaller or equal to that of a point to which this algorithm
would have converged if no perturbations were applied -- everything else being
equal.\textquotedblright\ 

Numerous works that are cited in \cite{sup-bib-online} show that this global
function reduction of the SM occurs in practice in many real-world
applications. But until the guarantee problem of the SM is answered one
wonders if the SM is just a successful heuristic or if there is a mathematical
foundation for the accumulating reports on its performance success? Therefore,
answering the guarantee problem of the SM is an intriguing issue, which to our
knowledge, has not been discussed in the literature in any way.

\textbf{Concentration of measure. }Concentration of measure (about a median)
is a principle that is applied in measure theory, probability and
combinatorics, and has consequences for other fields such as Banach space
theory. Informally, it states that \textquotedblleft A random variable that
depends in a Lipschitz way on many independent variables (but not too much on
any of them) is essentially constant\textquotedblright, \cite{talagrand96}.

The concentration of measure phenomenon was put forth in the early 1970s by
Vitali Milman in his works on the local theory of Banach spaces, extending an
idea going back to the work of Paul L\'{e}vy, as noted in \cite{gromov}. It
was further developed in the works of Milman and Gromov, Maurey, Pisier,
Schechtman, Talagrand \cite{talagrand96}, Ledoux \cite{Ledoux}, and others.

\textbf{Contribution and structure of this paper. }We offer an analysis of the
guarantee problem of the SM via the principle of concentration of measure.
This approach raises though some further questions but it is a first step
toward explaining why the SM works. In Section \ref{sec:background} we
elaborate on the SM while in Section \ref{sect:sm} we describe it in detail
and offer a layout of the situation in \textquotedblleft
matrix\textquotedblright\ form via an infinite lower triangular matrix called
the superiorization matrix. In Section \ref{sbsc:ConcMeas} we present a brief
primer on the principle of concentration of measure with which we intend to
analyze the behavior of the SM. The special case of linear superiorization
(LinSup) is discussed in Section \ref{sect:linsup}. A pathway to the nonlinear
case is discussed in Section \ref{sect:nonlinear}, followed by some concluding
remarks in Section \ref{sect:concluding}. Technical results that support and
enable the analysis are presented in the Appendices A.1--A.7 at the end of the paper.

\section{Background of the superiorization methodology\label{sec:background}}

\textbf{The superiorization methodology}. To answer in a succinct manner the
question \textquotedblleft what is the superiorization
methodology?\textquotedblright\ the next three paragraphs are quoted from our
preface to the special issue \textquotedblleft Superiorization: Theory and
Applications\textquotedblright\ \cite{special-issue}:

\textquotedblleft The superiorization methodology (SM) is used for improving
the efficacy of iterative algorithms whose convergence is resilient to certain
kinds of perturbations. Such perturbations are designed to `force' the
perturbed algorithm to produce more useful results for the intended
application than the ones that are produced by the original iterative
algorithm. The perturbed algorithm is called the `superiorized version' of the
original unperturbed algorithm. When the original algorithm is computationally
efficient and useful in terms of the application at hand and if the
perturbations are simple and not expensive to calculate, then the advantage of
this method is that, for essentially the computational cost of the original
algorithm, we are able to get something more desirable by steering its
iterates according to the designed perturbations. This is a very general
principle that has been used successfully in some important practical
applications, especially for inverse problems such as image reconstruction
from projections, intensity-modulated radiation therapy and nondestructive
testing, and awaits to be implemented and tested in additional fields.

An important case is when the original algorithm is `feasibility-seeking' (in
the sense that it strives to find some point that is compatible with a family
of constraints) and the perturbations that are introduced into the original
iterative algorithm aim at reducing (not necessarily minimizing) a given merit
function. In this case, superiorization has a unique place in optimization
theory and practice. Many constrained optimization methods are based on
methods for unconstrained optimization that are adapted to deal with
constraints. Such is, for example, the class of projected gradient methods
wherein the unconstrained minimization inner step `leads' the process and a
projection onto the whole constraint set (the feasible set) is performed after
each minimization step in order to regain feasibility. This projection onto
the entire constraints set is in itself a non-trivial optimization problem and
the need to solve it in every iteration hinders projected gradient methods and
restricts their efficiency only to feasible sets that are `simple to project
onto.' Barrier or penalty methods likewise are based on unconstrained
optimization combined with various `add-on's that guarantee that the
constraints are preserved. Regularization methods embed the constraints into a
`regularized' objective function and proceed with unconstrained solution
methods for the new regularized objective function.

In contrast to these approaches, the superiorization methodology can be viewed
as an antipodal way of thinking. Instead of adapting unconstrained
minimization algorithms to handling constraints, it adapts feasibility-seeking
algorithms to reduce merit function values. This is done while retaining the
feasibility-seeking nature of the algorithm and without paying a high
computational price. Furthermore, general-purpose approaches have been
developed for automatically superiorizing iterative algorithms for large
classes of constraints sets and merit functions; these provide algorithms for
many application tasks.\textquotedblright\ (end of quote.)

\textbf{Usefulness of the approach.} The usefulness of the SM relies on two
features: (i) \textbf{Computational}: feasibility-seeking is logically a
less-demanding task than seeking a constrained minimization point in a
feasible set. Therefore, letting efficient feasibility-seeking algorithms
\textquotedblleft lead\textquotedblright\ the algorithmic effort and modifying
them with inexpensive add-ons works well in practice. (ii)
\textbf{Applicational}: in some significant real-world applications the choice
of a target function is exogenous to the modeling and data collection which
give rise to the constraints. In such situations the limited confidence in the
usefulness of a chosen target function leads often to the recognition that,
from the application-at-hand point of view, there is no need, neither a
justification, to search for an exact constrained minimum\footnote{Some
support for this reasoning may be borrowed from the American scientist and
Noble-laureate Herbert Simon who was in favor of \textquotedblleft
satisficing\textquotedblright\ rather than \textquotedblleft
maximizing\textquotedblright. Satisficing is a decision-making strategy that
aims for a satisfactory or adequate result, rather than the optimal solution.
This is because aiming for the optimal solution may necessitate needless
expenditure of time, energy and resources. The term \textquotedblleft
satisfice\textquotedblright\ was coined by Herbert Simon in 1956 \cite{simon},
see: https://en.wikipedia.org/wiki/Satisficing.}. For obtaining
\textquotedblleft good results\textquotedblright,\ evaluated by how well they
serve the task of the application at hand, it is often enough to find a
feasible point that has reduced (not necessarily minimal) target function value.

\textbf{Weak superiorization\ and strong superiorization}. It is worthwhile to
note here that there are two research directions in the general area of the
superiorization methodology. One is the direction when only bounded
perturbation resilience is used and the constraints are assumed to be
consistent (having nonempty intersection). Then, one treats the
\textquotedblleft superiorized version\textquotedblright\ of the original
unperturbed basic algorithm actually as a recursion formula that produces an
infinite sequence of iterates, and convergence questions are meant in their
asymptotic nature. This is the framework in which we work in this paper. The
second direction does not assume consistency of the constraints but uses
instead a proximity function that \textquotedblleft measures\textquotedblright%
\ the violation of the constraints. Instead of seeking asymptotic feasibility,
it looks at $\varepsilon$-compatibility with $C$ and uses the notion of
\textquotedblleft strong perturbation resilience\textquotedblright, see
\cite[Subsection II.C]{med-phys-2012} where this direction has been initiated.
The same core \textquotedblleft superiorized version\textquotedblright\ of the
original unperturbed algorithm might be investigated in each of these
directions, but the second is the more useful one for practical applications,
whereas the first makes only asymptotic statements. \ The terms
\textquotedblleft weak superiorization\textquotedblright\ and
\textquotedblleft strong superiorization\textquotedblright\ were proposed as a
nomenclature for the first and second directions, respectively, in
\cite[Section 6]{cz3-2015} and \cite{constanta-weak-strong}. We do not discuss
here the latter, therefore, whenever we say superiorization in the sequel we
mean weak superiorization.

\section{The guarantee problem of the superiorization
methodology\label{sect:sm}}

In order to consider basic algorithms of the form (\ref{eq:basic-alg}) that
are resilient to perturbations as those that appear in (\ref{eq:perturbed-alg}%
) formally, the following definition is used, see, e.g., \cite[Definition
1]{cdh10}, where it was formulated for a single algorithmic operator, i.e.,
$A_{t}=A$ for all $t\geq0$.

\begin{definition}
\label{definition-BPR}\textbf{\textit{Bounded Perturbation Resilience (BPR)}}
Given a sequence of operators $A_{t}:{{\mathcal{H}}}\rightarrow{{\mathcal{H}}%
},$ for all $t\geq0,$ an algorithm as in (\ref{eq:basic-alg}) is said to be
\texttt{bounded perturbations resilient} if the following holds: If the
algorithm (\ref{eq:basic-alg}) generates sequences $(x_{n})_{n}$ that converge
to points in $C$ for all $x_{0}\in{{\mathcal{H}}},$ then any sequence
$(x_{n}^{\prime})_{n}$, generated by (\ref{eq:perturbed-alg}) where the vector
sequence $(v_{n})_{n}$ is bounded, $\beta_{n}\geq0$ for all $n\geq0$, and
${\sum_{n=0}^{\infty}}\beta_{n}<+\infty,$ also converges to a point in $C$ for
any $x_{0}^{\prime}\in{{\mathcal{H}}}$.
\end{definition}

These notions appear in earlier papers on the SM, see, e.g.,
\cite{cdh10,dhc-itor-09,compare2014,med-phys-2012}.

In addition to the basic algorithm and its superiorized version we consider in
the SM a target function $\phi:\Delta\rightarrow E$, whose domain
$\Delta\subseteq E^{J}$ contains the feasible set $C,$ and we adopt the
convention that a point in $\Delta$ for which the value of $\phi$ is smaller
is considered \textit{superior} to a point in $\Delta$ for which the value of
$\phi$ is larger. The essential idea of the SM is to make use of the
perturbations of (\ref{eq:perturbed-alg}) to transform a perturbation
resilient algorithm that seeks a feasible solution (the basic algorithm) into
its superiorized version whose outputs are equally good from the point of view
of feasibility-seeking, but are superior (not necessarily optimal) with
respect to the target function $\phi$.

The SM, which works well in numerous numerical applications (consult
\cite{sup-bib-online}), consists of choosing the perturbation vectors $v_{n}$
in (\ref{eq:perturbed-alg}) as directions of nonascent of $\phi$ in the
superiorized version of the basic algorithm. With the above information we
formulate the guarantee problem of the SM.

\begin{problem}
\textbf{The guarantee problem of weak superiorization}

The \texttt{guarantee problem of the weak superiorization method}, discussed
here, is the following question: Can we provably guarantee, maybe under some
assumptions, that for a given nonempty constraints set $C$ of a CFP and a
target function $\phi:\Delta\rightarrow E$ such that $C\subseteq\Delta$ we
will have $\phi(x_{\infty}^{\prime})\leq\phi(x_{\infty})$ for the limits
$x_{\infty}$ and $x_{\infty}^{\prime}$ of sequences $(x_{n})_{n}$ and
$(x_{n}^{\prime})_{n}$ generated by the basic algorithm (\ref{eq:basic-alg})
and its superiorized version (\ref{eq:perturbed-alg}), respectively, both
initiated at the same starting point $x_{0}=x_{0}^{\prime}$?
\end{problem}

To the best of our knowledge, this has not been answered in any way. The only
result in this direction is the attempt to investigate the behavior of a
superiorized version of a basic feasibility-seeking algorithm done by us in
\cite{cz3-2015}. The main result there (Theorem 4.1 in \cite{cz3-2015})
establishes a mathematical basis for the behavior of the SM when dealing with
input data of constrained minimization problems, i.e., a target function and a
constraints set. In particular, a feasible region that is the intersection of
finitely many closed convex constraint sets is assumed. The dynamic
string-averaging projection (DSAP) method, with variable strings and variable
weights, is playing there the role of a feasibility-seeking algorithm, which
is indeed bounded perturbations resilient. The bounded perturbations
resilience of the DSAP method has been proved in \cite{cz-2013} and it is
worthwhile to note that the DSAP is an algorithmic scheme that includes
several well-known specific feasibility-seeking algorithms as special cases.
These include, but are not limited to, the sequential Kaczmarz projections
method and the simultaneous Cimmino projections method, see, e.g.,
\cite{cegielski-book}.

Theorem 4.1 in \cite{cz3-2015} says that any sequence, generated by the
superiorized version of a DSAP algorithm (Algorithm 4.1. there), will not only
converge to a feasible point, a fact which is due to the bounded perturbations
resilience of the DSAP method, but, additionally, that exactly one of two
alternatives must hold. Either its limit point will solve the constrained
minimization problem of the same data, or that the sequence is strictly
Fej\'{e}r monotone with respect to (i.e., gets strictly closer to the points
of) a subset of the solution set of the constrained minimization problem of
the same data. But Fej\'{e}r monotonicity, even if strict, does not yield
convergence to a point in the set with respect to which the sequence is
strictly Fej\'{e}r monotone. So, this result shows that one gets closer to a
subset of the solution set of the constrained minimization problem but it
falls short of proving the convergence toward such a set.

The superiorization method uses input data consisting of a constraints set $C$
which is the intersection of several individual sets $C_{1},C_{2},\ldots
,C_{I}$ and a target function $\phi$. Feasibility-seeking with a sequential
projections basic algorithm will lead asymptotically to a feasible point
$x^{\ast}.$ Perturbations via interlaced local moves in the negative gradient
direction will not prevent the process from converging to a feasible point if
the basic feasibility-seeking algorithm is bounded perturbations resilient.
Convergence of the superiorized algorithm to any superior feasible point is
the subject of the \textquotedblleft guarantee problem of SM\textquotedblright%
\ discussed in this paper. Any superior feasible point has a target function
value $\phi$ that is lower than that of the feasible point $x^{\ast}$ which is
reached (asymptotically) by the same basic feasibility-seeking algorithm
without any interlaced perturbations -- everything else in the implementation,
such as relaxation parameters, initialization point, ordering of the
individual sets that are projected on, etc. -- being equal.

\subsection{A layout of the SM as a matrix\ of elements\label{sbsc:Matrix}}

We will use the following definitions.

\begin{definition}
Let $A:{\mathcal{H}}\rightarrow{\mathcal{H}}$ be an operator and let
$D\subset{\mathcal{H}}.$

(i) The operator $A$ is called \texttt{nonexpansive} on $D$ if
\begin{equation}
\Vert A(x)-A(y)\Vert\leq\Vert x-y\Vert,{\text{\ for all\ }}x,y\in D.
\end{equation}

(ii) The operator $A$ is called \texttt{monotone} on $D$ if
\begin{equation}
\langle y-x,Ay-Ax\rangle\geq0,{\text{\ for all\ }}x,y\in D.
\end{equation}

\end{definition}

These definitions describe the action of $A$ on a pair $x,y$ compared with the
original pair: Nonexpansive operators do not make the pair \textquotedblleft
further apart\textquotedblright, while monotone operators \textquotedblleft do
not rotate it in more than $90$ degrees.\textquotedblright\ A linear
orthogonal projection is nonexpansive and monotone, any linear operator with
norm $\leq1$ is nonexpansive while any linear operator whose symmetric part is
positive definite is monotone. Also, the nearest point projection on a closed
convex set is nonexpansive and monotone, see, e.g., \cite[Example
20.12]{BC-book}. To facilitate our analysis we define an infinite lower
triangular matrix\ of elements of ${{\mathcal{H}}}$ and name it \textit{the
Superiorization Matrix}.

\begin{definition}
\label{def:sm-matrix}(\textbf{The Superiorization Matrix}) Let $(A_{t}%
)_{t=1}^{\infty},$ $(x_{n})_{n},$ $(v_{n})_{n},$ and $(\beta_{n})_{n}$ be as
in the previous section above. Define $\left(  {{\mathcal{M}}}(n,k)\right)
_{n=0,k=0}^{\infty,\infty}$, an infinite lower triangular matrix of elements
of ${{\mathcal{H}}}$ as follows:

(1) In the upper left corner define an arbitrary vector in ${{\mathcal{M}}%
}(0,0):=x_{0}=x_{0}^{\prime}\in{{\mathcal{H}}}$.

(2) Construct the $n$-th row from the $(n-1)$-th row by applying $A_{n}$, in
each column $k$, to the entry above it in the $(n-1)$-th row in that column:
\begin{equation}
{{\mathcal{M}}}(n,k):=A_{n}({{\mathcal{M}}}(n-1,k)),\qquad k=0,1,\ldots,n.
\end{equation}

(3) Additionally, add for each $n\geq0,$ the $(n,n+1)$-th entry which is
obtained from the $(n,n)$-th entry by adding to it $\beta_{n}v_{n}$:
\begin{equation}
{{\mathcal{M}}}(n,n+1):={{\mathcal{M}}}(n,n)+\beta_{n}v_{n}.
\end{equation}

\end{definition}

The superiorization matrix can be described in the following form:\bigskip%
\begin{equation}
\left[
\begin{array}
[c]{cccccc}
& {\tiny 0} & {\tiny 1} & {\tiny 2} & {\tiny \cdots} & {\tiny n}\\
{\tiny 0} & {\tiny x}_{0}={\tiny x}_{0}^{\prime} & {\tiny x}_{0}{\tiny +\beta
}_{0}{\tiny v}_{0} &  &  & \\
{\tiny 1} & {\tiny x}_{1}{\tiny =A}_{1}{\tiny (x}_{0}{\tiny )} & {\tiny x}%
_{1}^{\prime}{\tiny =A}_{1}{\tiny (x}_{0}{\tiny +\beta}_{0}{\tiny v}%
_{0}{\tiny )} & {\tiny x}_{1}^{\prime}{\tiny +\beta}_{1}{\tiny v}_{1} &  & \\
{\tiny 2} & {\tiny x}_{2}{\tiny =A}_{2}{\tiny (x}_{1}{\tiny )} & {\tiny A}%
_{2}{\tiny (x}_{1}^{\prime}{\tiny )} & {\tiny x}_{2}^{\prime}{\tiny =A}%
_{2}{\tiny (x}_{1}^{\prime}{\tiny +\beta}_{1}{\tiny v}_{1}{\tiny )} &  & \\
{\tiny \vdots} & {\tiny \vdots} & {\tiny \vdots} & {\tiny \vdots} &
{\tiny \ddots} & {\tiny x}_{n-1}^{\prime}{\tiny +\beta}_{n-1}{\tiny v}_{n-1}\\
{\tiny n} & {\tiny x}_{n}{\tiny =A}_{n}{\tiny (x}_{n-1}{\tiny )} &
{\tiny \cdots} & {\tiny \cdots} & {\tiny \cdots} & {\tiny x}_{n}^{\prime
}{\tiny =A}_{n}{\tiny (x}_{n-1}^{\prime}{\tiny +\beta}_{n-1}{\tiny v}%
_{n-1}{\tiny )}\\
{\tiny \vdots} & \boldsymbol{\downarrow} & \boldsymbol{\downarrow} &
\boldsymbol{\downarrow} & {\tiny \cdots} & \boldsymbol{\downarrow}\\
& {\tiny x}_{\infty,0}{\tiny =x}_{\infty}{\tiny \in C} & {\tiny x}_{\infty
,1}{\tiny \in C} & {\tiny x}_{\infty,2}{\tiny \in C} & {\tiny \cdots} &
{\tiny x}_{\infty,n}{\tiny \in C}%
\end{array}
\right]  . \label{eq:sup-matrix}%
\end{equation}

The upper-most row and left-hand side column include the column and row
indices of the matrix, respectively. The bottom-row is not part of the matrix
either but depicts the limits of the sequences of each column. Only the first
$n$ columns are depicted but the matrix has infinitely many columns as well as
infinitely many rows. The sequence in the $0$-th column is generated by the
basic algorithm, thus, converges to $x_{\infty,0}=x_{\infty}\in C,$ while the
sequence of the main diagonal elements of the matrix are the iterates
generated by the superiorized version of the basic algorithm which, therefore,
converges to $x_{\infty}^{\prime}.$

This matrix representation of the SM is new and has never been published
before. We consider it an indispensable tool in analyzing the progress of
iterative sequences generated by the SM.

\begin{lemma}
\textit{The Superiorization Matrix} of Definition \ref{def:sm-matrix} has the
following properties:

(i) For all $n\geq0$, ${{\mathcal{M}}}(n,0)=x_{n}$.

(ii) The infinite sequence of the elements in the $k$-th column, for each
$k\geq0,$ converges to a point $x_{\infty,k}\in C.$ Observe that $x_{\infty
,0}=x_{\infty}.$

(iii) The diagonal elements of the matrix are ${{\mathcal{M}}}(n,n)=x_{n}%
^{\prime}$, thus, if bounded perturbations resilience holds then the infinite
sequence of the elements along the main diagonal of the superiorization matrix
will converge to $x_{\infty}^{\prime}.$

(iv) For a target function $\phi$ whose domain contains the feasible set $C$
\begin{equation}
\phi(x_{n})-\phi(x_{n}^{\prime})=\sum_{k=1}^{n}\phi({{\mathcal{M}}%
}(n,k-1))-\phi({{\mathcal{M}}}(n,k)),{\text{ For all }}n\geq0. \label{eq:phi}%
\end{equation}

\end{lemma}

\begin{proof}
(i) This follows from the definition. It means that the infinite sequence of
all elements in the $0$-th column constitute the sequence $(x_{n})_{n}$
generated by the basic (feasibility-seeking) algorithm.

(ii) This is so because in each column only a finite number of initial
elements are perturbed and from one point onward the operators are applied
without further perturbations. Recall that we assumed that any sequence
$(x_{n})_{n},$ generated by the basic algorithm (\ref{eq:basic-alg}),
converges to a point in $C,$ for any initial point$.$

(iii) This follows by induction since ${{\mathcal{M}}}(0,1)=x_{0}+\beta
_{0}v_{0}$; ${{\mathcal{M}}}(1,1)=A_{1}(x_{0}+\beta_{0}v_{0})=x_{1}^{\prime}$;
${{\mathcal{M}}}(1,2)=x_{1}^{\prime}+\beta_{1}v_{1}$; ${{\mathcal{M}}%
}(2,2)=A_{2}(x_{1}^{\prime}+\beta_{1}v_{1})=x_{2}^{\prime}$, and so on.

(iv) This follows from $\phi(x_{n})=\phi({{\mathcal{M}}}(n,0))$ and
$\phi(x_{n}^{\prime})=\phi({{\mathcal{M}}}(n,n))$ by going along the $n$-th
row of the matrix.
\end{proof}

The following lemma states that if a basic algorithm like (\ref{eq:basic-alg})
always converges to a point in $C$ and if the operators $A_{t}:{{\mathcal{H}}%
}\rightarrow{{\mathcal{H}}},$ for all $t\geq0,$ are nonexpansive, then the
superiorized version of the basic algorithm also converges to a point in $C$.

\begin{lemma}
Let $(A_{t})_{t=1}^{\infty}$ be a sequence of operators $A_{t}:{{\mathcal{H}}%
}\rightarrow{{\mathcal{H}}}$ that gives rise to an iterative process which,
starting from any initial $x_{0}\in{{\mathcal{H}}},$ generates a sequence
$(x_{n})_{n}\subset{{\mathcal{H}}}$ by (\ref{eq:basic-alg}) and assume that
any sequence $(x_{n})_{n},$ generated by this process, converges to some point
$x_{\infty}\in C.$ If the operators $(A_{t})_{t=1}^{\infty}$ are nonexpansive
and if $v_{n}\in{{\mathcal{H}}}$ and $\beta_{n}$ are real numbers so that
$\Vert v_{n}\Vert\leq M$, $\beta_{n}\geq0,$ for all $n\geq0,$ and $\sum
_{n=0}^{\infty}\beta_{n}<+\infty$ then the algorithm (\ref{eq:basic-alg}) is
bounded perturbations resilient.
\end{lemma}

\begin{proof}
In each column, from some row downward, only consecutive applications of the
operators $A_{t}$ occur. Therefore, since any sequence generated by the basic
algorithm always converges, every column $k$ converges to some limit
$x_{\infty,k}$ $\in C$. Dividing all $v_{n}$ by their norms we can assume,
without loss of generality, that $\Vert v_{n}\Vert=1,$ for all $n\geq0.$ Thus,
looking at the main diagonal entries, we have
\begin{equation}
\Vert{{\mathcal{M}}}(k,k+1)-{{\mathcal{M}}}(k,k)\Vert=\Vert\beta_{k}v_{k}%
\Vert=\beta_{k},
\end{equation}
and, due to the nonepansiveness of the operators, all distances between any
other pair of horizontally-neighboring entries in the $s$-th and $(s+1$)-th
columns are smaller or equal $\beta_{s}$, hence, we have also that neighboring
column limits are close, i.e., $\Vert x_{\infty,s+1}-x_{\infty,s}\Vert
\leq\beta_{s}$ for all $s\geq0$. Therefore, the sequence $(x_{\infty,k})_{k}$
of all limits of the columns is a Cauchy sequence of elements in $C$, which
will converge to some $x_{\infty,\infty}\in C$. To show that the latter is the
limit of $(x_{n})_{n}$, observe that for every row $n$ and for all $k<n$, we
have
\begin{equation}
\Vert x_{n}^{\prime}-{{\mathcal{M}}}(n,k)\Vert=\Vert{{\mathcal{M}}%
}(n,n)-{{\mathcal{M}}}(n,k)\Vert\leq\sum_{s=k}^{n}\beta_{s}\leq\sum
_{s=k}^{\infty}\beta_{s}.
\end{equation}
Since $\lim_{n\rightarrow\infty}{{\mathcal{M}}}(n,k)=x_{\infty,k}$, the
distance of any limit point of $(x_{n}^{\prime})$ from $x_{\infty,k}$ must be
smaller or equal $\sum_{s=k}^{\infty}\beta_{s}$, hence, has distance smaller
or equal $2\sum_{s=k}^{\infty}\beta_{s}$ from $x_{\infty,\infty}$. But
$\lim_{k\rightarrow\infty}\left(  2\sum_{s=k}^{\infty}\beta_{s}\right)  =$ $0$
which yields $\lim_{n\rightarrow\infty}x_{n}^{\prime}=x_{\infty,\infty}$.
\end{proof}

This lemma should be compared with Theorem 1 of \cite{med-phys-2012}. The
latter makes more assumptions and proves strong perturbation resilience, not
only bounded perturbation resilience. So, these two results complement each other.

\section{Concentration of measure\label{sbsc:ConcMeas}}

The phenomenon of \textit{concentration of measure} is the fact that, in some
important cases of random variables, it turns out that with almost full
probability the random variable is very close to its expectation, aka mean.
For example, a classical case of concentration of measure is the Law of Large
Numbers, see, e.g., \cite{seneta}, combined with the Central Limit Theorem of
probability theory, which describe how with almost full probability a
\textit{sum of many independent random variables} is concentrated near its
mean. Moreover, the distribution of the sum is almost normal. The literature
on this topic is wide and varied, see, e.g.,
\cite{com-book,conv-concent-book,samson}.

To explain the principle of concentration of measure in a manner that is
appropriate for our needs, we focus on a case, featuring in high-dimensional
Euclidean spaces,\emph{ }i.e., $E^{N}$ with the Euclidean norm $\Vert
\cdot\Vert$ and inner product $\langle\cdot,\cdot\rangle$ -- the
$N$-dimensional real Hilbert space. There, for uniform probability in its unit
sphere $S^{N-1}:=\{u\in E^{N}\,|\,\Vert u\Vert=1\}$, almost the whole mass
concentrates near the equator. In other words, for randomly given two vectors,
with almost full probability the angle between them is near $90^{\circ}$.

More precisely, fix a unit vector $u_{0}\in S^{N-1},$ e.g., $u_{0}%
=(1,0,\ldots,0)$. For $u\in S^{N-1}$, let $\alpha$ be the latitude relative to
$u_{0},$ i.e., the angle between $u$ and the hyperplane orthogonal to $u_{0}$,
so, $-\frac{1}{2}\pi\leq\alpha\leq\frac{1}{2}\pi$. The uniform measure in
$S^{N-1}$ can be \textquotedblleft disintegrated\textquotedblright\ along
$\alpha$, to levels $\alpha=const.$ which are translates of $\cos\alpha\cdot
S^{N-2}$. Therefore, if we denote by $d\omega=d\omega_{N-2}$ the uniform (say,
normalized to be probability) measure on $S^{N-2}$ then the uniform measure on
$S^{N-1}$ will be
\begin{equation}
K\cdot(\cos\alpha)^{N-2}\,d\omega\,d\alpha,
\end{equation}
where $K$ is a normalizing constant. For $N$ large, $(\cos\alpha)^{N-2}$ has a
steep peak near $\alpha=0$, thus, almost the whole mass is concentrated there.
Indeed, for $\alpha$ small, which will, thus, be the significant case,
\begin{equation}
(\cos\alpha)^{N-2}\approx(1-\frac{1}{2}\alpha^{2})^{N-2}\approx
\operatorname{exp}(-\frac{1}{2}N\alpha^{2}),
\end{equation}

i.e., the distribution of $\alpha$, for big $N$, is very close to normal
distribution with standard deviation $1/\sqrt{N}$.

The concentration of measure principle may be derived also in an alternative
way, where the uniform distribution on the sphere $S^{N-1}$ is treated, very
usefully, as follows. Take the distribution on $x\in E^{N}$ with coordinates
i.i.d.\ (independent identically distributed) $\sim{\mathcal{N}}={\mathcal{N}%
}(0,1)$, i.e.,\ distributed as standard normal -- with mean $0$ and standard
deviation $1$ (that is, according to $(1/\sqrt{2\pi}) \operatorname{exp}%
(-\frac{1}{2}x^{2})\,dx$.)

As is well-known, this distribution in $E^{N}$ is invariant under any
orthogonal self-map of\emph{ }$E^{N}$. This means that $u=x/\Vert x\Vert$ will
be distributed uniformly on $S^{N-1}$. So, we have here a vehicle to get this
uniform distribution. This also implies that $\langle x,a\rangle
\sim{\mathcal{N}}$, \ for any fixed unit vector $a$.

If one considers $\sum_{i=1}^{N}\eta(x_{i})$, with any function $\eta$, then
the distribution of that sum will lose the orthogonal symmetry, but since the
$\eta(x_{i}),\,\,i=1,2,\ldots,N,$ are still independent, the Law of Large
Numbers and the Central Limit Theorem still apply. Thus, the distribution of
the sum is concentrated near its expectation.

This applies, in particular, to $\Vert x\Vert_{2}=\left(  \sum_{i=1}^{N}%
x_{i}^{2}\right)  ^{1/2}$. Its expectation is $N$, since the expectation of a
single $x_{i}^{2}$ is the variance which is equal to the square of the
standard deviation, thus, equal to $1$. And we recapture the main assertion
above: As $\Vert x\Vert/\sqrt{N}$ is near $1$ with almost full probability,
the distribution of $\langle u,a\rangle,$ for $u$ uniform on $S^{N-1},$ (e.g.,
our $u=x/\Vert x\Vert$), is very near $(1/\sqrt{N})\cdot\langle x,a\rangle$ --
the latter standard normal. In particular, $\langle u,a\rangle$ is very
unlikely to be different from zero more than in an order of magnitude of
$1/\sqrt{N}$.

We shall make use of some facts, in spirit of concentration of measure, which
arise in high dimensional Euclidean (i.e.,\ real Hilbert) $E^{N}$, which are
derived in Appendix \ref{app:ConcMeas} at the end of this paper.

\section{The case of linear superiorization (LinSup)\label{sect:linsup}}

Linear superiorization (LinSup) was investigated in \cite{linsup-ip17,
vladivostok} where a linear setting is considered. The operators of the basic
algorithm are projections on half-spaces, thus, involve linear projections on
hyperplanes plus constants, and the target function $\phi$ is linear, i.e.,
$\phi(x):=\left\langle c,x\right\rangle +a$ where $c$ is a given vector and
$a$ is a given real constant.

In the superiorization matrix ${\mathcal{M}}$ (Definition \ref{def:sm-matrix}%
{) setting, these operators act on the pairs along the neighboring $i$-th and
$(i+1)$-th columns, in particular, these operators are rotating and
stretching/shrinking the \textquotedblleft increments\textquotedblright%
\ }$\Delta_{k,i}$ defined by
\begin{equation}
\Delta_{k,i}:={\mathcal{M}}{(k,i+1)-{\mathcal{M}}(k,i).}%
\end{equation}
{\ }

To handle this, the idea is to treat the operators as a random sample. Since
what the operators do to increments does not depend on the constant part, we
characterize the operators by the unit vector $u$ orthogonal to the bounding
hyperplane of each half-space, and assume that these vectors are a sample from
a uniform distribution on $S^{N-1}$.

Then, by the principle of concentration of measure, with almost full
probability, $u$ will be almost orthogonal to the increment in question,
indeed making angle $\frac{1}{2}\pi+\alpha$ where $\alpha$ is distributed in
almost a normal distribution with standard deviation $1/\sqrt{N}$. The
hyperplane orthogonal to $u$, onto which $A_{n}$ projects, will make that
small angle $\alpha$ with the increment, thus, the projection of that
increment -- the increment in the next row -- is rotated in that small angle
$\alpha$ (and has almost the same length.)

In other words, by the principle of concentration of measure in
high-dimensional spaces that we speak of, if one has an instance of our
operator acting on a vector (in our case - an increment) $y$, it would be a
very unexpected \textquotedblleft anomaly\textquotedblright\ not to find $u$
and $y$ to be almost orthogonal -- making an angle $\frac{1}{2}\pi+\alpha$
with $\alpha$ small as above, thus, to have the hyperplane orthogonal to $u$
making that small angle $\alpha$ with $y$. All these arguments are true
provided that we are justified to use our probabilistic model (i.e., with\ $u$
distributed uniformly).

Thus, the application of the linear operator $A_{k}$ in the passage from the
$(k-1)$-th to the $k$-th row downward along the neighboring $i$-th and
$(i+1)$-th columns, the increment $\Delta_{k-1,i}$ becomes
\begin{equation}
A_{k}\Delta_{k-1,i}=A_{k}({\mathcal{M}}(k-1,i+1)-A_{k}({\mathcal{M}%
}(k-1,i))={\mathcal{M}}(k,i+1)-{\mathcal{M}}(k,i)=\Delta_{k,i},
\end{equation}
in fact adding to it an \textquotedblleft alteration\textquotedblright\ which
is, with almost full probability, normed relatively $O(1/\sqrt{N})$ of it.

In adding these alterations when moving from the $i$-th row (where the
increment was $\beta_{i}v_{i}$) to the $n$-th row where we would use
(\ref{eq:phi}), one may, with almost full probability, use Conclusion
\ref{conc:Sum} in Appendix \ref{sbsc:Norm} below,\textbf{ }to find that the
relative accumulated alteration is $O(\sqrt{n-i}/\sqrt{N})$.

Yet, as long as that relative accumulated alteration does not approach $1$, we
can be sure that the increment at the $n$-th row has less than $90^{\circ}$
angle with the original \textquotedblleft good\textquotedblright\ direction
$v_{i}$. Thus, the pair will be \textquotedblleft good\textquotedblright%
\ (i.e., $\phi$ will decrease along it), since, $\phi$ being affine, the
direction of decrease does not depend on the point in space, and we will be done.

So, we should be safe, with almost full probability, as long as $n$ (the
number of steps the algorithm has taken before being stopped) does not
approach $N$. Then we may very well expect to find that $\phi(x_{\infty
}^{\prime})\leq\phi(x_{\infty})$.

We conjecture that such considerations should give us more than the desired
inequality $\phi(x_{\infty}^{\prime})\leq\phi(x_{\infty})$. We should be able
to estimate quantitatively how much $\phi(x_{\infty}^{\prime})$ is less than
$\phi(x_{\infty})$ (with almost full probability), but we are unable to do so
at this time.

\section{The nonlinear case: A potential pathway\label{sect:nonlinear}}

\subsection{A multi-dimensional \textquotedblleft mean-value\textquotedblright%
\ fact\label{subsect:mean-v-fact}}

Let $X$, $Y$ be real Banach spaces and let $F:U\subset X\rightarrow Y$ be a
$C^{1}$ function from an open subset $U$ in $X$ to $Y$. Let $x_{0}$ and
$x_{1}$ be two points in $U$, such that the line-segment connecting them is
contained in $U$. Write $w:=x_{1}-x_{0}$.

Then,
\begin{equation}
F(x_{1})-F(x_{0})=\int_{0}^{1}\dfrac{d}{dt}F(x_{0}+tw)\,dt=\int_{0}%
^{1}DF(x_{0}+tw)_{w}\,dt. \label{eq:int}%
\end{equation}
Thus, the vector $F(x_{1})-F(x_{0})$ belongs to the closed convex hull of the
set of values of the operator $DF$ that is the derivative operator of $F$
computed at the points $x$ on the segment connecting $x_{0}$ and $x_{1}$, and
applied to $w$ denoted by $DF(x)_{w}$.

This means that in order to bound an \textquotedblleft
increment\textquotedblright\ $F(x_{1})-F(x_{0})$, in reference to
$w=x_{1}-x_{0}$, we may as well, for $C^{1}$ functions, bound the value
$DF(x)_{w}$ that the operator $DF(x)$ takes on $w$ for $x$ along the
line-segment connecting $x_{0}$ and $x_{1}$.

\subsection{Computing the derivative of the projection on a convex
set\label{sbsc:DerProj}}

Differentiability of the metric projection operator onto a convex set has been
studied in the literature, see, e.g., \cite{silhavy-2015} and references
therein. We develop this here in a self-contained manner suitable to our
needs. Let $C$ be a closed convex subset of a Hilbert space ${\mathcal{H}}$
and let $P$ be the nearest-point (metric) projection operator onto $C$. We
wish to compute the derivative operator $DP(x)$. For that we assume that $C$
has smooth boundary $\partial C$ (in the general case $\partial C$ might be
approximated by a smooth one) and assume that $x\notin C$.

Often in the literature one investigates conditions for such a projection to
be differentiable, in one or another sense, for general convex $C$, which is
not always the case and is a subtle question, e.g., \cite{shapiro-2016}. Here
we concentrate on computing the formula for the operator derivative. We do not
detail here justifications from Differential Geometry.

Let $x$ be a point in the complement of $C,$ and let $\bar{x}:=P(x)$ be the
point on the (assumed smooth) boundary $\partial C$ of $C$, at which $C$ has a
tangent (affine) hyperplane $\bar{x}+H$, which is the translation of some
(linear) hyperplane $H,$ which, of course, depends on $x$. Since $\bar{x}$ is
the nearest point to $x$ in $C$, we have the orthogonality relation
$x-P(x)\bot H$. Also, $d(x):=\Vert x-P(x)\Vert$ is the distance from $x$ to
$C$, and along the line-segment connecting $x$ to $P(x),$\thinspace
\thinspace$P$ is constant, equal to $P(x)$. Therefore, the operator derivative
$DP(x)_{w}$ vanishes on the line through $x$ and $P(x),$ i.e., for all $w\in
R(x-P(x))$.

We still have to compute $DP(x)$ on the orthogonal complement hyperplane $H$.
Let ${\mathcal{C}}:=\{x\in{\mathcal{H}}\mid d(x)=c\}$ be the \textquotedblleft
hypersurface\textquotedblright\ of points at constant distance $c>0$ from $C$,
passing through $x$.

\begin{claim}
The tangent hyperplane to ${\mathcal{C}}$ at $x$ is $H$.
\end{claim}

\begin{proof}
Indeed, for every $x$ on the hypersurface ${\mathcal{C}}$\thinspace\thinspace
we have $d^{2}(x)=\langle x-P(x),x-P(x)\rangle=c^{2}$. Differentiating this,
we find for any $w$ in the tangent hyperplane to ${\mathcal{C}}$ at
$x$,\thinspace\thinspace that $\langle w-DP(x)_{w},x-P(x)\rangle=0$. This
means that $(w-DP(x)_{w})\in H$. But the image of $P$ is contained in the
boundary of $C$, hence $DP(x)_{w}\in H,$ for all $w$, and we find that if $w$
is in the tangent hyperplane to ${\mathcal{C}}$ then $w\in H$, which proves
the claim.\bigskip
\end{proof}

So, our task of finding $DP(x)_{w}$ for $w\in H$ boils down to computing the
operator derivative, from $H$ to itself, of $P|_{{\mathcal{C}}}$, the
restriction of $P$ to ${\mathcal{C}}$.

Looking at the inverse mapping $Q$ of $P|_{{\mathcal{C}}}$ we see that on
points $y$ at the boundary of $C$, $Q(y)=y+d\cdot\vec{n}(y),$ where $\vec
{n}(y)$ is the outer unit normal to the boundary of $C$ at $y$, and $d$ is the
constant value of the distance on ${\mathcal{C}}$.

But the operator derivative of $\vec{n}$ is, by definition, the curvature
operator $\kappa$ from $H$ to itself, which is a positive-definite symmetric
operator, with principal axes and eigenvalues that are the directions and
values of principal curvatures, respectively, see, e.g., \cite[Chapters 1 and
7]{Lee-book}. In extreme (limiting) cases these are $0$ for flat and $\infty$
for an angle. Thus, $DQ(y)=\boldsymbol{1}+d\cdot\kappa$, $\boldsymbol{1}$
denoting the identity operator, and for the inverse $DP(x)=(\boldsymbol{1}%
+d\cdot\kappa)^{-1}$ on $H$. All the above leads to, and proves, the following lemma.

\begin{lemma}
The operator derivative of $P(x)$ at some $x\notin C$, in the case of smooth
$C$, is a positive-definite symmetric operator, equal to $0$ on $R(x-P(x))$
and equal to $(\boldsymbol{1}+d\cdot\kappa)^{-1}$ on $H=R(x-P(x))^{\bot}$,
$\kappa$ being the curvature operator for $\partial C$ at $\bar{x}=P(x)$.
Thus, $DP(x)$ is between $0$ and $1$.
\end{lemma}

By the way, this immediately implies, by our \textquotedblleft
Mean-Value\textquotedblright\ Fact in Subsection \ref{subsect:mean-v-fact}
that, in the smooth $\partial C$ case, (otherwise one may approximate $C$ by a
smooth) $P$ is nonexpansive and monotone. This is a well-known fact, that is
usually proved in the literature in other ways. See, e.g., \cite[Fact
1.5]{bb96} for nonexpansivness and \cite[Example 20.12]{BC-book} for
monotonicity of $P,$ respectively.

\subsection{Toward the nonlinear case}

For the nonlinear case the situation is more complicated. Here the operators
$A_{n}$ are projections onto convex sets. Recall the superiorization
matrix\ (Subsection \ref{sbsc:Matrix}). To compare $\phi(x_{n})=\phi
({\mathcal{M}}(n,0))$ with $\phi(x_{n}^{\prime})=\phi({\mathcal{M}}(n,n))$, we
add, as in (\ref{eq:phi}), the \textquotedblleft increments\textquotedblright%
\ at the $(n,i)$-th and $(n,i+1)$-th entry, these coming from moving along the
columns by applying the $A_{n}$ operators and then applying a $\phi$ value
reduction step.

Our task is basically to assess increments. By the \textquotedblleft
Mean-Value\textquotedblright\ Fact in Subsection \ref{subsect:mean-v-fact}, we
may instead assess the result of operator derivatives $DP(x)$ acting
successively, and then $\nabla\phi$, on the original difference $\beta
_{i}v_{i}$ that we had between the $(i,i+1)$-th and the $(i,i)$-th entries.
Indeed, the summands in (\ref{eq:phi}) are their integrals as in (\ref{eq:int}).

By Subsection \ref{sbsc:DerProj}, this cascade of $DP(x)$'s operates as
follows: each of them first projects its argument $w$ onto $H$
($H=(x-P(x))^{\bot}$, is, of course, a function of $x$). By our above
principle of concentration of measure, $w$ is very unlikely not to be almost
orthogonal to the normal of $H$, i.e.,\ to form a small angle $\alpha$ with
$H$, where $\alpha$ is distributed almost $\sim{\mathcal{N}}(0,1/\sqrt{N})$
($N$ is the dimension of the Euclidean space $E^{N}$). But, contrary to the
linear case, the projected part is then subjected to the action of
$(\boldsymbol{1}+d\cdot\kappa)^{-1}$. Indeed, in the linear case the curvature
operator $\kappa$ is always equal $0$ and $(\boldsymbol{1}+d\cdot\kappa
)^{-1}=\boldsymbol{1}$.

In order to achieve our goal to have \textquotedblleft good\textquotedblright%
\ increments along the $n$-th row, it would be good if the result of applying
successively the cascade of operators on $\beta_{i}v_{i}$ makes an angle
smaller or equal $90^{\circ}$ with $-\nabla\phi(x)$ \emph{computed at the
final point $x$} (while we chose the $v_{i}$ in some way to be OK at the
initial point). That might be hampered both by the deviations caused by the
operator derivatives -- the $\alpha$ and the effect of $\kappa$, and by the
change in $\nabla\phi$ between the initial and final points. We address these
issues, in the light of \textquotedblleft concentration of
measure\textquotedblright\ conclusions of Appendices \ref{sbsc:Sum},
\ref{sbsc:LinOp} and \ref{sbsc:RotMat} that are at the end of the paper.
Specifically, we try to bound, for our path down the column of the
superiorization matrix (\ref{eq:sup-matrix}),

(1) How much the vector is rotated by the $DP$'s -- the effects both of
$\alpha$ and $\kappa$,

(2) By how much its norm has decreased, and

(3) How much the place to compute $\nabla\phi$ \textquotedblleft
moved\textquotedblright\ from the initial to the final point in ${\mathcal{H}%
}$.

First, by Lemma \ref{conc:Sum} the distance between $x$ (where we chose
$v_{i}$) and the result of applying the cascade of $P$'s to it (and where we
should compute $\nabla\phi$) is supposed to be near the square root of the sum
of the distances $d=\Vert x-P(x)\Vert$ along the way from the $i$-th to the
$n$-th stage. Thus, it is small in the final stages when the $d$'s are small
(indeed, they are converging to $0$).

As for the effect of the $\alpha$ alteration by the projections, the situation
is as in the linear case -- we should be safe as long as $n$ does not approach
the dimension $N$ of the Euclidean space $E^{N}$.

For the accumulated terms $({\mathbf{1}}+d\cdot{\mathbf{\kappa}})^{-1}$ along
the path (in what follows we denote by $k$ indices along the path, i.e.,
$k\in\operatorname*{path})$ denote the eigenvalues (here, also the singular
values) of the encountered ${\mathbf{\kappa}}_{k}$ (the curvature operator in
the hyperplane $H$), i.e.,\ the relevant principal curvatures, by
$(\kappa_{\ell}^{(k)})_{\ell},$ for $\ell=1,2,\ldots,N-1$. Then those of
$({\mathbf{1}}+d\cdot{\mathbf{\kappa}})^{-1}$ are $\left(  (1+d_{k}\cdot
\kappa_{\ell}^{(k)})^{-1}\right)  _{\ell}$, so that, by Conclusion
\ref{conc:ActMat}, and using the $\Vert\cdot\,\Vert_{p}^{(\pi)}$ norm of
Appendix \ref{sbsc:Norm} below, for $(N-1)$-dimensional vectors, their product
is expected to multiply the norm of the vector they act upon by
\begin{equation}
\prod_{k\in\operatorname*{path}}\left\Vert \left(  (1+d_{k}\cdot\kappa_{\ell
}^{(k)})^{-1}\right)  _{\ell}\right\Vert _{2}^{(\pi)},\label{eq:Norm}%
\end{equation}
still with relative deviation of the order of at most $O(\sqrt{n-i}/\sqrt{N})$.

By Conclusion \ref{conc:ProdMat} in Appendix \ref{sbsc:RotMat}, they are
expected to rotate the direction of the vector, i.e.,\ shift the normalized
vector, by
\begin{equation}
\sqrt{2\left(  1-\prod_{k\in\operatorname*{path}}\dfrac{\left\Vert \left(
(1+d_{k}\cdot\kappa_{\ell}^{(k)})^{-1}\right)  _{\ell}\right\Vert _{1}^{(\pi
)}}{\left\Vert \left(  (1+d_{k}\cdot\kappa_{\ell}^{(k)})^{-1}\right)  _{\ell
}\right\Vert _{2}^{(\pi)}}\right)  }. \label{eq:Devi}%
\end{equation}
with relative deviation of the order of at most $O(\sqrt{n-i}/\sqrt{N})$.

Observe that $\Vert\,\cdot\Vert_{1}^{(\pi)}$ $\leq$ $\Vert\,\cdot\Vert
_{2}^{(\pi)}$ (cf.\ Appendix \ref{sbsc:Norm}) and, by\ Remark
\ref{remark:ProdMat}, the value of (\ref{eq:Devi}) is always $\leq\sqrt{2}$,
meaning angle of rotation $\leq90^{\circ}$. Indeed, in many cases it will be
much less than $90^{\circ}$. For example, for vectors $\left(  (1+d_{k}%
\cdot\kappa_{\ell}^{(k)})^{-1}\right)  _{\ell}$ with equal (resp.\ almost
equal) entries (in our case -- either \textquotedblleft
spherical\textquotedblright\ curvature or when the $d_{k}\cdot\kappa$ are
small), the $\Vert\cdot\,\Vert_{2}^{(\pi)}$ norm will be equal (resp.\ almost
equal) to the $\Vert\cdot\,\Vert_{1}^{(\pi)}$ norm, hence the terms in the
product in (\ref{eq:Devi}) will be near $1$.

Both (\ref{eq:Norm}) and (\ref{eq:Devi}) refer to the $(N-1)$-dimensional
vectors $v=((1+d_{k}\cdot\kappa_{\ell}^{(k)})^{-1})_{\ell}$, having entries in
$(0,1]$. In (\ref{eq:Norm}), which controls how much the norm was reduced, we
have the product of $\Vert v\Vert_{2}^{(\pi)}$. In (\ref{eq:Devi}), which
controls how much the direction was rotated, we have the square root of twice
$1$ minus the product of $\Vert v\Vert_{1}^{(\pi)}/\Vert v\Vert_{2}^{(\pi)}$.

\begin{proposition}
For an $(N-1)$-dimensional vector $v=(v_{\ell})_{\ell}$ with components
$v_{\ell}\in(0,1]$, we have
\begin{equation}
\left(  \Vert v\Vert_{2}^{(\pi)}\right)  ^{2}\leq\Vert v\Vert_{1}^{(\pi)}%
\leq\frac{1}{2}\left(  \left(  \Vert v\Vert_{2}^{(\pi)}\right)  ^{2}+1\right)
.
\end{equation}

\end{proposition}

\begin{proof}
Since $v_{\ell}\in(0,1]$, one has $v_{\ell}^{2}\leq v_{\ell}$. Averaging, we
get $\left(  \Vert v\Vert_{2}^{(\pi)}\right)  ^{2}\leq\Vert v\Vert_{1}^{(\pi
)}$. Also, by definition of $\Vert v\Vert_{2}^{(\pi)}$ for $(N-1)$-dimensional
vectors, see Appendix \ref{sbsc:Norm},
\begin{align}
&  \left(  \Vert v\Vert_{2}^{(\pi)}\right)  ^{2}=\dfrac{1}{N-1}\sum_{\ell
=1}^{N-1}v_{\ell}^{2}=1-\dfrac{1}{N-1}\sum_{\ell=1}^{N-1}(1-v_{\ell}^{2})\\
&  =1-\dfrac{1}{N-1}\sum_{\ell=1}^{N-1}(1-v_{\ell})(1+v_{\ell})\geq
1-2\dfrac{1}{N-1}\sum_{\ell=1}^{N-1}(1-v_{\ell})\nonumber\\
&  =2\dfrac{1}{N-1}\sum_{\ell=1}^{N-1}v_{\ell}-1=2\Vert v\Vert_{1}^{(\pi)}-1.
\end{align}
Hence, $\Vert v\Vert_{1}^{(\pi)}\leq\frac{1}{2}(\Vert v\Vert_{2}^{(\pi)}%
)^{2}+1)$, which completes the proof.
\end{proof}

As a consequence of this proposition we have,
\begin{align}
&  \dfrac{\Vert v\Vert_{1}^{(\pi)}}{\Vert v\Vert_{2}^{(\pi)}}\geq\dfrac{(\Vert
v\Vert_{2}^{(\pi)})^{2}}{\Vert v\Vert_{2}^{(\pi)}}=\Vert v\Vert_{2}^{(\pi
)},\nonumber\\
&  \dfrac{\Vert v\Vert_{1}^{(\pi)}}{\Vert v\Vert_{2}^{(\pi)}}\leq\frac{1}%
{2}\dfrac{(\Vert v\Vert_{2}^{(\pi)})^{2}+1}{\Vert v\Vert_{2}^{(\pi)}}=\frac
{1}{2}\left(  \Vert v\Vert_{2}^{(\pi)}+\dfrac{1}{\Vert v\Vert_{2}^{(\pi)}%
}\right)  ,
\end{align}

So, there is here a \emph{\textquotedblleft balancing effect\textquotedblright%
} -- if the angle of rotation becomes close to $90^{\circ}$ in (\ref{eq:Devi}%
), then the norm will be reduced considerably in (\ref{eq:Norm}). Thus, when
$i$ is such that $d_{i}$ times a \textquotedblleft typical\textquotedblright%
\ curvature ${\mathbf{\kappa}}$ (loosely, the ratio between $d$ and a
\textquotedblleft typical\textquotedblright\ radius of the $C_{i}$) is still
considerably larger than $1$ (maybe while in the early columns of the
superiorization matrix with $i$ small), then, by (\ref{eq:Norm}), the cascade
of $DP$ will reduce the norm hugely, hence, anyway applying $\nabla\phi$ then
will give a negligible result.

On the other hand, when we reach a stage where $d_{i},d_{i+1},\ldots,d_{n}$
are small, both the possible rotation and the distance traveled are
controlled. But of course, then the decrease of the $\beta_{k}$ should also be
taken into account. For big $i$, thus small $\beta_{i}$, the contribution
might again be negligible. This shows that the main contribution in
(\ref{eq:phi}) seems to come from intermediate terms.

As said above, the angle of rotation, both by the $\alpha$ and by the $\kappa$
seems to be controlled, as long as the number of steps $n$ does not approach
the vector space dimension $N$. If conditions are imposed on the target
function $\phi$ then point (3) above could also be tackled, in view of the
preceeding paragraph, bringing our analysis closer to conclusion.

\section{Concluding comments\label{sect:concluding}}

We explored here the fundamental open problem of the superiorization method
which is the question under what conditions one can guarantee that a
superiorized version of a bounded perturbation resilient feasibility-seeking
algorithm converges to a feasible point that has target function value smaller
or equal to that of a point to which this algorithm would have converged if no
perturbations were applied -- everything else being equal.

The success of the superiorization method in many real-world applications, as
witnessed in \cite{sup-bib-online}, made this an important question. However,
in the absence of a conclusive deterministic argument, we applied here the
probabilistic principle of concentration of measure. For linear
superiorization (LinSup) this approach works quite well whereas our analysis
aimed at using it for a general nonlinear situation is still less
conclusive.\bigskip\appendix

\section{Some concentration of measure facts in a high-dimensional $E^{N}%
$\label{app:ConcMeas}}

\subsection{The probability $L^{p}$ norms of vectors\label{sbsc:Norm}}

For a vector $x\in E^{N}$, and $1\leq p<\infty$, denote by $\Vert\,\cdot
\Vert_{p}^{(\pi)}$ ($\pi$ stands for \textquotedblleft probability
space\textquotedblright) its $L^{p}$ norm when the set of indices
$\{1,2,\ldots,N\}$ is made into a uniform probability space, giving each index
a weight $1/N$, namely
\begin{equation}
\Vert x\Vert_{p}^{(\pi)}:=\left(  \dfrac{1}{N}\sum_{j=1}^{N}|x_{j}%
|^{p}\right)  ^{1/p}, \label{eq:Lp}%
\end{equation}
see, e.g., \cite{song97}. As with any probability measure, always
$\Vert\,\cdot\Vert_{p}^{(\pi)}$ increases with $p$.

For $x_{1},x_{2},\ldots,x_{N}$ i.i.d.\ $\sim{\mathcal{N}}$, $\left(  \Vert
x\Vert_{p}^{(\pi)}\right)  ^{p}$ is an average: its expectation ${\mathbb{E}}$
will be the same as the expectation of $|x|^{p}$ for $x$ a scalar distributed
$\sim{\mathcal{N}}$:
\begin{equation}
{\mathbb{E}}\left[  |x|^{p}\right]  =\dfrac{1}{\sqrt{2\pi}}\int|x|^{p}%
\operatorname{exp}(-\textstyle{\frac{1}{2}}x^{2})\,dx,
\end{equation}
but its standard deviation will be $1/\sqrt{N}$ that of $|x|^{p}$ for a scalar
$\sim{\mathcal{N}}$:
\begin{equation}
\dfrac{1}{\sqrt{N}}\dfrac{1}{\sqrt{2\pi}}\int\left(  |x|^{p}-{\mathbb{E}%
}\left[  |y|^{p}\right]  \right)  ^{2}\operatorname{exp}(-\textstyle{\frac
{1}{2}}x^{2})\,dx.
\end{equation}
Thus, $\Vert x\Vert_{p}^{(\pi)}$ is highly concentrated around the, not
depending on $N$, $\left(  {\mathbb{E}}\left[  |x|^{p}\right]  \right)
^{1/p}$ with degree of concentration $O(1/\sqrt{N})$.

One may conclude, loosely speaking, that in any case, these $\Vert\,\cdot
\Vert_{p}^{(\pi)}$ norms, having not depending on $N$ means, are expected to
be $O(1)$, for all $N$.

\subsection{The norm of the sum of vectors with given norms\label{sbsc:Sum}}

Suppose we are given $M$ vectors $y_{1},y_{2},\ldots,y_{M}$ of known norms
$d_{1},d_{2},\ldots.d_{M}$ in $E^{N}$. What should we expect the norm of their
sum to be?

This can be answered: take the direction of each of them distributed uniformly
on $S^{N-1}$, \emph{even conditioned on fixed valued for the others}. In other
words, take them independent, each with direction distributed uniformly. This
can be constructed by taking random $M$ vectors in $E^{N}$ (that is, a random
$M\times N$ matrix), with entries i.i.d.\ $\sim{\mathcal{N}}$, dividing them
by $\sqrt{N}$, then by their norm (now highly concentrated near $1$) and
multiplying them by $d_{1},d_{2},\ldots,d_{M},$ respectively.

The sum $\sum_{i=1}^{M}y_{i}$, if we ignore the division by the norm, is
$1/\sqrt{N}$ times the random matrix applied to the vector $(d_{1}%
,d_{2},\ldots,d_{M})$. But the distribution of the random matrix is invariant
with respect to any transformation which is orthogonal with respect to the
Hilbert-Schmidt norm -- the square root of the sum of squares of the entries
(i.e.,\ $\Vert T\Vert_{HS}:=\sqrt{\operatorname*{tr}(T^{\prime}\cdot T)}%
=\sqrt{\operatorname*{tr}(T\cdot T^{\prime})}$,\thinspace\thinspace$T^{\prime
}$ denoting the transpose and $\operatorname*{tr}$ standing for the trace,
see, e.g., \cite{bell16}). In particular, the distribution of the sum is the
same as that of $1/\sqrt{N}$ times $\sqrt{d_{1}^{2}+d_{2}^{2}+\cdots+d_{M}%
^{2}}$ times the random matrix applied to $(1,0,\ldots,0)$, which is, of
course, distributed with independent $\sim{\mathcal{N}}$ entries, thus, with
norm concentrated near $\sqrt{N}$. (With relative deviation $O(1/\sqrt{N})$.)
This leads to the following conclusion.

\begin{conclusion}
\label{conc:Sum} For $M$ vectors $y_{1},y_{2},\ldots,y_{M}$ of known norms
$d_{1},d_{2},\ldots.d_{M}$, in $E^{N}$ we have that $\Vert\sum_{i=1}^{M}%
y_{i}\Vert$ is near $\sqrt{d_{1}^{2}+d_{2}^{2}+\cdots+d_{M}^{2}}$ with almost
full probability (With relative deviation $O(1/\sqrt{N})$.)
\end{conclusion}

\subsection{The accumulation of given distances on the unit
sphere\label{sbsc:Sphere}}

As in the previous Appendix \ref{sbsc:Sum}, we seek to find what should we
expect the norm of a sum of $M$ vectors of given norms $d_{1},d_{2}%
,\ldots.d_{M}$ to be. But here the vectors are \emph{the differences between
consecutive elements in a sequence of points on the unit sphere}
$S^{N-1}\subset E^{N}$. Denote by $\omega_{N-1}$ the normalized to be
probability (i.e.,\ of total mass $1$) uniform measure on $S^{N-1}$.

\begin{remark}
By symmetry, for $x=(x_{1},x_{2},\ldots,x_{N})\in S^{N-1}$, $\int x_{k}%
^{2}\,d\omega_{N-1}$ is the same for all $k$. Of course, their sum is
$\int1\,d\omega_{N-1}=1$. Therefore,
\begin{equation}
\int x_{k}^{2}\,d\omega_{N-1}=\dfrac{1}{N},\qquad k=1,2,\ldots,N.
\label{eq:integral}%
\end{equation}
Hence, for a polynomial of degree $\leq2$ on $E^{n}$:
\begin{equation}
p(x)=\langle Qx,x\rangle+2\langle a,x\rangle+\gamma, \label{eq:poly}%
\end{equation}
where $Q$ is a symmetric $N\times N$ matrix, $a\in E^{N}$ and $\gamma\in E$,
we will have
\begin{equation}
\int p(x)\,d\omega_{N-1}=\dfrac{1}{N}\operatorname*{tr}\,Q+\gamma.
\end{equation}

\end{remark}

Note that, for some fixed $0\leq d\leq2$, the set of points in $S^{N-1}$ of
distance $d$ from some fixed vector $u\in S^{N-1}$ is the $(N-2)$-sphere
$\subset S^{N-1},$ $\Sigma(u,d)$ given by
\begin{equation}
\Sigma(u,d):=(1-d^{2}/2)u+d\sqrt{1-d^{2}/4}\cdot{S^{N-2}}_{u^{\bot}%
},\label{eq:sphere}%
\end{equation}
where ${S^{N-2}}_{u^{\bot}}$ stands for the unit sphere in the hyperplane
prependicular to $u.$ In our scenario, one performs a \emph{Markov chain}$,$
see, e.g., \cite{markov2000}\emph{.} Starting from a point $u_{0}$ on
$S^{N-1}$, and moving to a point $u_{1}\in\Sigma(u_{0},d_{1})$ uniformly
distributed there. Then, from that $u_{1}$, to a point $u_{2}\in\Sigma
(u_{1},d_{2})$ uniformly distributed there, and so on, until one ends with
$u_{M}$. We would like to find ${\mathbb{E}}\left[  \Vert u_{M}-u_{0}%
\Vert\right]  $.

If we denote by ${\mathcal{L}}_{d}$ the operator mapping a function $p$ on
$S^{N-1}$ to the function whose value at a vector $u\in S^{N-1}$ is the
average of $p$ on $\Sigma(u,d)$, then ${\mathcal{L}}_{d_{k}}(p)$ evaluated at
$u$ is the expectation of $p$ at the point to which $u$ moved in the $k$-th
step above. Hence, in the above Markov chain, the expectation of $p(u_{M})$
is
\begin{equation}
{\mathcal{L}}_{d_{M}}{\mathcal{L}}_{d_{M-1}}\cdots{\mathcal{L}}_{d_{1}}%
p(u_{0}).
\end{equation}
Thus, what we are interested in is
\begin{equation}
{\mathbb{E}}\left[  \Vert u_{M}-u_{0}\Vert\right]  ={\mathcal{L}}_{d_{M}%
}{\mathcal{L}}_{d_{M-1}}\cdots{\mathcal{L}}_{d_{1}}(\Vert x-u_{0}\Vert
^{2})(u_{0}).
\end{equation}

So, let us calculate ${\mathcal{L}}_{d}(p)$ for polynomials of degree $\leq2$
as in (\ref{eq:poly}). In performing the calculation, assume $u=(1,0,\ldots
,0)$. For $x=(x_{1},x_{2},\ldots,x_{N})\in E^{N}$ write $y=(x_{2},x_{3}%
,\ldots,x_{N})\in E^{N-1}$. In (\ref{eq:poly}) write $a=(a_{1},b)$ where
$b=(a_{2},a_{3},\ldots a_{N})\in E^{N-1}$ and
\begin{equation}
Q=\left(
\begin{array}
[c]{cc}%
\eta & c^{\prime}\\
c & Q^{\prime}%
\end{array}
\right)  ,
\end{equation}
where $Q^{\prime}$ is a symmetric $(N-1)\times(N-1)$ matrix, $c\in E^{N-1}$
and $\eta\in E$. Note that for our $u=(1,0,\ldots,0)$,\thinspace
\thinspace\thinspace$a_{1}=\langle a,u\rangle$,\thinspace\thinspace
\thinspace\ $\eta=\langle Qu,u\rangle$ and $\operatorname*{tr}\,Q^{\prime
}=\operatorname*{tr}\,Q-\eta=\operatorname*{tr}\,Q-\langle Qu,u\rangle$.

Then, for $p$ as in in (\ref{eq:poly}),
\begin{equation}
p(x)=\eta x_{1}^{2}+2x_{1}\langle c,y\rangle+\langle Q^{\prime}y,y\rangle
+2a_{1}x_{1}+2\langle b,y\rangle+\gamma.
\end{equation}
Hence, taking account of (\ref{eq:sphere}) for $u=(1,0,\ldots,0)$, and using
(\ref{eq:integral}),
\begin{align}
&  ({\mathcal{L}}_{d}p)(u)=({\mathcal{L}}p)(1,0,\ldots,0)\nonumber\\
&  =\left(  1-\dfrac{d^{2}}{2}\right)  ^{2}\eta+\dfrac{1}{N-1}d^{2}\left(
1-\dfrac{d^{2}}{4}\right)  \operatorname*{tr}\,Q^{\prime}+2\left(
1-\dfrac{d^{2}}{2}\right)  a_{1}+\gamma\nonumber\\
&  =\left(  1-\dfrac{d^{2}}{2}\right)  ^{2}\langle Qu,u\rangle+\dfrac{1}%
{N-1}d^{2}\left(  1-\dfrac{d^{2}}{4}\right)  \left(  \operatorname*{tr}%
Q-\langle Qu,u\rangle\right)  \nonumber\\
&  +2\left(  1-\dfrac{d^{2}}{2}\right)  \langle a,u\rangle+\gamma,
\end{align}
which, by symmetry, will hold for any $u\in S^{N-1}$. In particular, we find,
as should be expected, that
\begin{align}
&  \int({\mathcal{L}}_{d}(p))(x)\,d\omega_{N-1}\nonumber\\
&  =\dfrac{1}{N}\left(  1-\dfrac{d^{2}}{2}\right)  ^{2}\operatorname*{tr}%
\,Q+\dfrac{1}{N-1}d^{2}\left(  1-\dfrac{d^{2}}{4}\right)  \left(  1-\dfrac
{1}{N}\right)  \operatorname*{tr}\,Q+\gamma\nonumber\\
&  =\dfrac{1}{N}\operatorname*{tr}\,Q+\gamma=\int p(x)\,d\omega_{N-1}.
\end{align}
We are interested, for some fixed $u\in S^{N-1}$, in
\begin{equation}
p(x)=\Vert x-u\Vert^{2}=2(1-\langle u,x\rangle).
\end{equation}
Then there is no $Q$ term, so one has
\begin{equation}
\Big ({\mathcal{L}}_{d}\big (2(1-\langle u,x\rangle)\big )\Big )(u)=2\left(
1-\left(  1-\dfrac{d^{2}}{2}\right)  \langle u,x\rangle\right)  .
\end{equation}
Consequently,
\begin{align}
&  {\mathbb{E}}\left[  \Vert u_{M}-u_{0}\Vert^{2}\right]  =\left(
{\mathcal{L}}_{d_{M}}{\mathcal{L}}_{d_{M-1}}\cdots{\mathcal{L}}_{d_{1}}(\Vert
x-u_{0}\Vert^{2})\right)  (u_{0})\nonumber\\
&  =2\left(  1-\left.  \left(  {\mathcal{L}}_{d_{M}}{\mathcal{L}}_{d_{M-1}%
}\cdots{\mathcal{L}}_{d_{1}}(\langle u_{0},x\rangle)\right)  \right\vert
_{x=u_{0}}\right)  \nonumber\\
&  2\left.  \left(  1-\prod_{i=1}^{M}\left(  1-\dfrac{d_{i}^{2}}{2}\right)
\langle u_{0},x\rangle\right)  \right\vert _{x=u_{0}}=2\left(  1-\prod
_{i=1}^{M}\left(  1-\dfrac{d_{i}^{2}}{2}\right)  \right)  .
\end{align}
This is $O\left(  M\cdot\left(  \Vert(d_{1},d_{2},\ldots,d_{M})\Vert_{2}%
^{(\pi)}\right)  ^{2}\right)  $. We also assess the standard deviation, which
is%
\begin{equation}
=2\sqrt{\left.  {\mathcal{L}}_{d_{M}}{\mathcal{L}}_{d_{M-1}}\cdots
{\mathcal{L}}_{d_{1}}(\langle u_{0},x\rangle^{2})\right\vert _{x=u_{0}%
}-\left(  \prod_{i=1}^{M}\left(  1-\dfrac{d_{i}^{2}}{2}\right)  \right)  ^{2}%
}.
\end{equation}
Here $p(x)=\langle a,x\rangle^{2}$, so there is only the $Q$ term with
$Q(x):=\langle a,x\rangle^{2}$. Then $\operatorname*{tr}\,Q=\Vert a\Vert^{2}$,
and we find
\begin{align}
&  \Big ({\mathcal{L}}_{d}\big (\langle a,x\rangle^{2}%
\big )\Big )(u)\nonumber\\
&  =\left(  1-\dfrac{d^{2}}{2}\right)  ^{2}\langle a,u\rangle^{2}+\dfrac
{1}{N-1}d^{2}\left(  1-\dfrac{d^{2}}{4}\right)  \left(  \Vert a\Vert
^{2}-\langle a,u\rangle^{2}\right)  \nonumber\\
&  =\left(  1-\dfrac{N}{N-1}d^{2}\left(  1-\dfrac{d^{2}}{4}\right)  \right)
\langle a,u\rangle^{2}+\dfrac{1}{N-1}d^{2}\left(  1-\dfrac{d^{2}}{4}\right)
\Vert a\Vert^{2}.
\end{align}
Consequently, for $a=u_{0}$ (note $\Vert u_{0}\Vert^{2}=1$)),
\begin{align}
&  \left.  \left(  {\mathcal{L}}_{d_{M}}{\mathcal{L}}_{d_{M-1}}\cdots
{\mathcal{L}}_{d_{1}}(\langle u_{0},x\rangle^{2})\right)  \right\vert
_{x=u_{0}}-\left(  \prod_{i=1}^{M}\left(  1-\dfrac{d_{i}^{2}}{2}\right)
\right)  ^{2}\nonumber\\
&  =-\left(  \prod_{i=1}^{M}\left(  1-\dfrac{d_{i}^{2}}{2}\right)  \right)
^{2}+\prod_{i=1}^{M}\left(  1-\dfrac{N}{N-1}d_{i}^{2}\left(  1-\dfrac
{d_{i}^{2}}{4}\right)  \right)  \nonumber\\
&  +\dfrac{1}{N-1}\left[  d_{1}^{2}\left(  1-\dfrac{d_{1}^{2}}{4}\right)
+d_{2}^{2}\left(  1-\dfrac{d_{2}^{2}}{4}\right)  \left(  1-\dfrac{N}{N-1}%
d_{1}^{2}\left(  1-\dfrac{d_{1}^{2}}{4}\right)  \right)  \right.  \nonumber\\
&  +d_{3}^{2}\left(  1-\dfrac{d_{3}^{2}}{4}\right)  \left(  1-\dfrac{N}%
{N-1}d_{2}^{2}\left(  1-\dfrac{d_{2}^{2}}{4}\right)  \right)  \left(
1-\dfrac{N}{N-1}d_{1}^{2}\left(  1-\dfrac{d_{1}^{2}}{4}\right)  \right)
\nonumber\\
&  \left.  +\cdots+{\text{ }}d_{M}^{2}\left(  1-\dfrac{d_{M}^{2}}{4}\right)
\prod_{i=1}^{M}\left(  1-\dfrac{N}{N-1}d_{i}^{2}\left(  1-\dfrac{d_{i}^{2}}%
{4}\right)  \right)  \right]  .
\end{align}
This is $O\left(  (M/N)\cdot\left(  \Vert(d_{1},d_{2},\ldots,d_{M})\Vert
_{4}^{(\pi)}\right)  ^{4}\right)  $, since the constant terms and the terms
with $d_{k}^{2}$ cancel, and the terms which do not cancel are coefficiented
by $O(1/N)$. Therefore, twice its square root, the standard deviation, will be
$O\left(  \sqrt{M}/\sqrt{N}\left(  \Vert(d_{1},d_{2},\ldots,d_{M})\Vert
_{4}^{(\pi)}\right)  ^{2}\right)  $, making the relative deviation
$O(1/\sqrt{MN})$. This leads to the following conclusion.

\begin{conclusion}
\label{conc:SphereVectors} The square of the norm of the sum of $M$ vectors of
given norms $d_{1},d_{2},\ldots.d_{M}$, which are differences between
consecutive elements in a sequence of points on the unit sphere $S^{N-1}%
\subset E^{N}$, modeled by the above Markov chain, is with almost full
probability, near%
\begin{equation}
2\left(  1-\prod_{i=1}^{M}\left(  1-\dfrac{d_{i}^{2}}{2}\right)  \right)  .
\end{equation}
(With relative deviation $O(1/\sqrt{MN})$.)
\end{conclusion}

\subsection{A reminder: Polar decomposition and singular values of a
matrix\label{sbsc:Decom}}

As is well-known, see, e.g., \cite{lange2010}, every fixed $N\times N$ matrix
$T$ can be uniquely written as $T=UA$ with $U$ orthogonal and $A$ symmetric
positive semidefinite (take $A=\sqrt{T^{\prime}\cdot T}$, then for every
vector $x$,\thinspace\thinspace$\Vert Tx\Vert$=$\Vert Ax\Vert$, so the map
$Ax\mapsto Tx$ is norm-preserving, i.e.,\ orthogonal), and also uniquely
written as $T=A_{1}U_{1}$ with $U_{1}$ orthogonal and $A_{1}$ symmetric
positive semidefinite (take $A_{1}=\sqrt{T\cdot T^{\prime}}$).

The \emph{singular values} of $T$ are defined as the eigenvalues of its
positive semidefinite part in the above decomposition. (It does not matter
from which side: $T\cdot T^{\prime}$ and $T^{\prime}\cdot T$ have the same
eigenvalues. Note that if $T$ is invertible they are similar: $T^{\prime
-1}(T\cdot T^{\prime})T$.)

Since any positive semidefinite matrix with eigenvalues $s_{1},s_{2}%
,\ldots,s_{N}$ is of the form
\begin{equation}
U^{\prime}\cdot\operatorname*{diag}(s_{1},s_{2},\ldots,s_{N})\cdot U
\end{equation}
with $U$ orthogonal ($\operatorname*{diag}$ denotes a diagonal matrix), we
find that the general form of a matrix with singular values $s_{1}%
,s_{2},\ldots,s_{N}$ is
\begin{equation}
T=U_{1}\cdot\operatorname*{diag}(s_{1},s_{2},\ldots,s_{N})\cdot U_{2},\qquad
U_{1}{\text{ and }}U_{2}{\text{ orthogonal}}. \label{eq:sing}%
\end{equation}

\subsection{Square matrix with entries independently $\sim{\mathcal{N}}$ and
the uniform distribution on orthogonals\label{sbsc:Mat}}

Take a random $N\times N$ matrix $Y$ with entries $Y_{i,j}$ i.i.d.\ $\sim
{\mathcal{N}}$. If we polarly decompose the random $Y$ as per Appendix
\ref{sbsc:Decom}, from either side, then the orthogonal part will be
distributed uniformly (i.e.,\ by Haar's measure) on the orthogonal group. This
follows from the fact that, by the symmetries of the above distribution of
$Y$, it is invariant under multiplying the random matrix on the right or left
by a fixed orthogonal matrix. So, we have here a vehicle to get this uniform
distribution. For a general excellent text on random matrices consult
\cite{edelman2005}.

For the positive semidefinite part we have to check, say, $Y^{\prime}\cdot Y$
for our random matrix $Y$. But if $u$ is any vector then, by the symmetries of
the distribution of the random $Y$, $Yu$ is distributed like $\Vert u\Vert$
times $Y\cdot(1,0,\ldots.0)$ -- i.i.d.\ $\sim{\mathcal{N}}$ entries, thus,
with norm concentrated near $\Vert u\Vert\cdot\sqrt{N}$, with relative
deviation $O(1/\sqrt{N})$. But, \emph{all the entries of $Y^{\prime}\cdot Y$
being discernible from $\langle Y^{\prime}\cdot Yu,u\rangle$ if we take as $u$
elements of the standard basis $e_{i}=(0,\ldots,0,1,0\ldots,0)$ and sums of
two of these}, we obtain the following conclusion.

\begin{conclusion}
\label{conc:Mat} $(1/N)Y^{\prime}\cdot Y$ (and likewise $(1/N)Y\cdot
Y^{\prime}$) is concentrated near ${\mathbf{1}}$ (${\mathbf{1}}$ denotes the
identity matrix), with relative deviation $O(1/N)$.
\end{conclusion}

In other words, the random $Y$ is, with almost full probability, very near
$\sqrt{N}$ times an orthogonal matrix. Indeed. to check how orthogonal
$(1/\sqrt{N})Y$ is, note that the amount it distorts the inner product between
unit vectors $u$ and $v$ is
\begin{equation}
(1/N)\langle Yu,Yv\rangle-\langle u,v\rangle=\langle((1/N)Y^{\prime}\cdot
Y-{\mathbf{1}})u,v\rangle=O(1/N).
\end{equation}

\subsection{The action of a linear operator in a high-dimensional
space\label{sbsc:LinOp}}

Consider an $N\times N$ matrix $T$ with given singular values $s_{1}%
,s_{2},\ldots,s_{N}$ as in (\ref{eq:sing}). Let $T$ act on a unit vector $u$
with direction uniformly distributed over $S^{N-1}$. By (\ref{eq:sing}) this
is distributed, up to an orthogonal \textquotedblleft
rotation\textquotedblright\ of the space, the same as $S=\operatorname*{diag}%
(s_{1},s_{2},\ldots,s_{N})$ acting on such a vector.

But by Section \ref{sbsc:ConcMeas}, that would be almost as $S$ applied to
$(1/\sqrt{N})x$,\thinspace\thinspace$x$ with coordinates i.i.d.\ $\sim
{\mathcal{N}}$, which is, of course, a vector with independent coordinates but
the $j$-th coordinate distributed as $(1/\sqrt{N})s_{j}$ times ${\mathcal{N}}$.

Now, similarly to what we had in Section \ref{sbsc:ConcMeas}, the square of
the norm of $S\cdot(1/\sqrt{N})x$, which is $(1/N)\sum_{j=1}^{N}s_{j}^{2}%
x_{j}^{2}$ has mean
\begin{equation}
(1/N)\sum_{j=1}^{N}s_{j}^{2}=\left(  \Vert(s_{1},s_{2},\ldots,s_{N})\Vert
_{2}^{(\pi)}\right)  ^{2},
\end{equation}
around which it is concentrated -- its standard deviation being
\begin{equation}
\sigma\cdot\sqrt{(1/N^{2})\sum_{j=1}^{N}s_{j}^{4}}=(1/\sqrt{N})\sigma
\cdot\left(  \Vert(s_{1},s_{2},\ldots,s_{N})\Vert_{4}^{(\pi)}\right)  ^{2},
\end{equation}
where $\sigma$ is the standard deviation for $x^{2}$ when $x\sim{\mathcal{N}}%
$, namely,
\begin{equation}
\sigma=\sqrt{\frac{1}{\sqrt{2\pi}}\int(x^{2}-1)^{2}\operatorname{exp}%
(-\textstyle{\frac{1}{2}}x^{2})\,dx}=\sqrt{2}.
\end{equation}
By Appendix \ref{sbsc:Norm}, the relative deviation is, thus, expected, with
almost full probability, to be $O(1/\sqrt{N})$. Note that since $T=U_{1}%
\cdot\operatorname*{diag}(s_{1},s_{2},\ldots,s_{N})\cdot U_{2}$, the value
around which the norm of $T$ applied to a uniformly distributed unit vector is
concentrated is
\begin{equation}
\Vert(s_{1},s_{2},\ldots,s_{N})\Vert_{2}^{(\pi)}=(1/\sqrt{N})\Vert S\Vert
_{HS}=(1/\sqrt{N})\Vert T\Vert_{HS}.
\end{equation}
Dividing $T$ by that value, we get a $T$ with $(1/\sqrt{N})\Vert T\Vert
_{HS}=1$ which, with almost full probability, will approximately preserve the
norm. How \textquotedblleft orthogonal\textquotedblright\ will it be? Let us
see how $S$ distorts the inner product between $(1/\sqrt{N})x$ and
$(1/\sqrt{N})y$, all $2N$ coordinates of $x$ and $y$ i.i.d.\ $\sim
{\mathcal{N}}$. The mean of the square of the difference
\begin{equation}
\langle S(1/\sqrt{N})x,S(1/\sqrt{N})y\rangle-\langle(1/\sqrt{N})x,(1/\sqrt
{N})y\rangle
\end{equation}
is
\begin{align}
&  (1/N^{2}){\mathbb{E}}\left(  \sum_{j=1}^{N}s_{j}^{2}x_{j}y_{j}-\sum
_{j=1}^{N}x_{j}y_{j}\right)  ^{2}=(1/N^{2}){\mathbb{E}}\left(  \sum_{j=1}%
^{N}(s_{j}^{2}-1)x_{j}y_{j}\right)  ^{2}\nonumber\\
&  =(1/N)\left(  (1/N)\sum_{j=1}^{N}(s_{j}^{4}-2s_{j}^{2}+1)\right)
\nonumber\\
&  =(1/N)\left(  \left(  \Vert(s_{1},s_{2},\ldots,s_{N})\Vert_{4}^{(\pi
)}\right)  ^{4}-2\left(  \Vert(s_{1},s_{2},\ldots,s_{N})\Vert_{2}^{(\pi
)}\right)  ^{2}+1\right) \nonumber\\
&  =(1/N)\left(  \left(  \Vert(s_{1},s_{2},\ldots,s_{N})\Vert_{4}^{(\pi
)}\right)  ^{4}-1\right)  .
\end{align}
Consequently, $T$ is orthogonal, with almost full probability, up to
$O(1/\sqrt{N})$. This leads to the following conclusion.

\begin{conclusion}
\label{conc:ActMat} An $N\times N$ matrix $T$ with given singular values
$s_{1},s_{2},\ldots,s_{N}$, acting on a high-dimensional $E^{N}$, would be
expected to act, with almost full probability, as
\begin{equation}
\Vert(s_{1},s_{2},\ldots,s_{N})\Vert_{2}^{(\pi)}=(1/\sqrt{N})\Vert T\Vert_{HS}%
\end{equation}
times an orthogonal matrix, up to a relative deviation $O(1/\sqrt{N})$.
\end{conclusion}

\begin{remark}
Now we address a seeming mystery\ raised by Conclusion \ref{conc:ActMat}. That
conclusion seems to require that $(1/\sqrt{N})$ times the Hilbert-Schmidt norm
of the product of two matrices with singular values $(s_{1},s_{2},\ldots
,s_{N})$ and $(s_{1}^{\prime},s_{2}^{\prime}\ldots,s_{N}^{\prime}),$
respectively, be equal to the product of the same for the factors, i.e.,\ to
$\Vert(s_{1},s_{2},\ldots,s_{N})\Vert_{2}^{(\pi)}\cdot\Vert(s_{1},s_{2}%
,\ldots,s_{N})\Vert_{2}^{(\pi)}$, up to relative deviation $O(1/\sqrt{N})$. Is
that so?

Note that, by (\ref{eq:sing}), the $HS$-norm of the product is that of
$SUS^{\prime}$ where $S=\operatorname*{diag}(s_{1},s_{2},\ldots,s_{N}%
)$,\thinspace$S^{\prime}=\operatorname*{diag}(s_{1}^{\prime},s_{2}^{\prime
}\ldots,s_{N}^{\prime})$ and $U$ is orthogonal. So, if, up to an $O(1/\sqrt
{N})$ relative deviation, we model $U$ as $(1/\sqrt{N})Y$, $Y=\left(
Y_{i,j}\right)  _{i,j}$ as in Appendix \ref{sbsc:Mat}, then $SUS^{\prime
}=\left(  (1/\sqrt{N})s_{i}Y_{i,j}s_{j}^{\prime}\right)  _{i,j}$. The square
of $(1/\sqrt{N})$ times its $HS$-norm is $(1/N^{2})\sum_{i,j=1,1}^{N,N}%
s_{i}^{2}Y_{i,j}^{2}s_{j}^{\prime2}$, with mean indeed equal to the square of
$\Vert(s_{1},s_{2},\ldots,s_{N})\Vert_{2}^{(\pi)}\cdot\Vert(s_{1}^{\prime
},s_{2}^{\prime}\ldots,s_{N}^{\prime})\Vert_{2}^{(\pi)}$, and with standard
deviation $\sigma\cdot1/N$ times the square of $\Vert(s_{1},s_{2},\ldots
,s_{N})\Vert_{4}^{(\pi)}\cdot\Vert(s_{1}^{\prime},s_{2}^{\prime}\ldots
,s_{N}^{\prime})\Vert_{4}^{(\pi)}$.
\end{remark}

\subsection{The rotation effected by an operator and by a product of operators
in a high-dimensional space\label{sbsc:RotMat}}

Let $T$ be an an $N\times N$ matrix, and consider \emph{the amount of rotation
between $v$ and $Tv$}. The square of the distance between these vectors, both
normalized to norm $1$ will be
\begin{align}
&  \left\Vert \dfrac{Tv}{\Vert Tv\Vert}-\dfrac{v}{\Vert v\Vert}\right\Vert
^{2}=\left\langle \dfrac{Tv}{\Vert Tv\Vert}-\dfrac{v}{\Vert v\Vert},\dfrac
{Tv}{\Vert Tv\Vert}-\dfrac{v}{\Vert v\Vert}\right\rangle \nonumber\\
=  &  2-\dfrac{\langle Tv,v\rangle+\langle v,Tv\rangle}{\Vert Tv\Vert\Vert
v\Vert}=2\left(  1-\dfrac{\langle T^{(sym)}v,v\rangle}{\Vert Tv\Vert\Vert
v\Vert}\right)  ,
\end{align}
where $T^{(sym)}:=\textstyle{\frac{1}{2}}(T+T^{\prime})$ is the
\emph{symmetric part} of $T$. Note that ${\text{tr}}\,T^{(sym)}%
=\operatorname*{tr}\,T$. So, we are led to investigate the inner product
$\langle Ax,x\rangle$ for $A$ symmetric. Let $(s_{1},s_{2},\ldots,s_{N})$ be
its eigenvalues, then $A=U^{\prime}SU$ where $S=\operatorname*{diag}%
(s_{1},s_{2},\ldots,s_{N})$ and $U$ orthogonal. As we did above, we take
$v=(1/\sqrt{N})x$, and\thinspace\thinspace$x$ with coordinates i.i.d.\ $\sim
{\mathcal{N}}$. Then
\begin{equation}
\left\langle A(1/\sqrt{N})x,(1/\sqrt{N})x\right\rangle =(1/N)\langle
U^{\prime}SUx,x\rangle=(1/N)\langle SUx,Ux\rangle.
\end{equation}
But, $Ux$ being distributed like $x$, this will have the same distribution as%
\begin{equation}
(1/N)\langle Sx,x\rangle=(1/N)\sum_{j=1}^{N}s_{j}x_{j}^{2}, \label{eq:S}%
\end{equation}
which has mean $(1/N)\sum_{j=1}^{N}s_{j}=(1/N)\operatorname*{tr}\,A$ and
$(1/\sqrt{N})\sigma\Vert(s_{1},\ldots,s_{N})\Vert_{2}^{(\pi)}$ is its standard
deviation. Of course, if $A$ is positive semidefinite then the $s_{\ell}\geq0$
and the above mean is $\Vert(s_{1},s_{2},\ldots,s_{N})\Vert_{1}^{(\pi)}$. This
leads to the following conclusion.

\begin{conclusion}
\label{conc:RotMat} For $T$ with symmetric part with eigenvalues $(s_{1}%
,s_{2},\ldots,s_{N})$, the square of the distance between $v$ and $Tv$, both
normalized to norm $1$, is, with almost full probability, near (with deviation
$O(1/\sqrt{N})$)%
\begin{equation}
2\left(  1-\dfrac{(1/N)\operatorname*{tr}\,T}{(1/\sqrt{N})\Vert T\Vert_{HS}%
}\right)  =2\left(  1-\dfrac{(1/N)\operatorname*{tr}\,T}{\Vert(s_{1}%
,s_{2},\ldots,s_{N})\Vert_{2}^{(\pi)}}\right)  ,
\end{equation}
which, if the symmetric part of $T$ is positive-semidefinite, is equal to%
\begin{equation}
2\left(  1-\dfrac{\Vert(s_{1},s_{2},\ldots,s_{N})\Vert_{1}^{(\pi)}}%
{\Vert(s_{1},s_{2},\ldots,s_{N})\Vert_{2}^{(\pi)}}\right)  .
\end{equation}

\end{conclusion}

The next discussion will lead to a conclusion about a product $A_{M}%
A_{M-1}\cdots A_{1}$ of a sequence of \emph{symmetric} operators. Consider a
symmetric $A=U^{\prime}\operatorname*{diag}(s_{1},s_{2},\ldots,s_{N})U$ with
given $s_{1},s_{2},\ldots,s_{N}$. Take $U$ uniformly distributed on the
orthogonal group, which we model up to a relative deviation $O(1/\sqrt{N})$ by
$1/\sqrt{N}\cdot Y$, $Y=\left(  Y_{i,j}\right)  _{i,j}$ as in Appendix
\ref{sbsc:Mat}. Then%
\begin{equation}
A=USU^{\prime}\approx\left(  (1/N)\sum_{k=1}^{N}Y_{k,i}s_{k}Y_{k,j}\right)
_{i,j}.\label{eq:A}%
\end{equation}
Consequently,%
\begin{equation}
{\mathbb{E[}}A]\approx(1/N)\left(  \sum_{k=1}^{N}s_{k}\right)  \cdot
{\mathbf{1}}=(1/N)\operatorname*{tr}\,A\cdot{\mathbf{1}}.
\end{equation}
But here we cannot say, as we did in previous cases, that, with high
probability, $A$ would be near that average -- indeed they cannot be
\textquotedblleft near\textquotedblright\ since the eigenvalues of the average
are all $(1/N)\operatorname*{tr}\,A$ while those of $A$ are with full
probability $s_{1},s_{2},\ldots,s_{N}$.

To apply the considerations of Appendix \ref{sbsc:Sphere}, where one relies on
a Markov chain employing uniform distribution on spheres, we inquire what is
the distribution of $Av_{0}$, and of the difference vector $\left(
\dfrac{Av_{0}}{\Vert Av_{0}\Vert}-\dfrac{v_{0}}{\Vert v_{0}\Vert}\right)  $
for a fixed $v_{0}$, with $A$ random as in (\ref{eq:A}) above. To fix matters,
assume $v_{0}=(1,0,\ldots,0)$. As above, we have $Av_{0}=U^{\prime}SUv_{0}$.
where $S:=\operatorname*{diag}(s_{1},s_{2},\ldots,s_{N})$. Or, with $U$
replaced by $1/\sqrt{N}\cdot Y$, $Av_{0}\approx(1/N)Y^{\prime}SYv_{0}$. Write
$Y$ as $(w,Z)$ where $w$ is the $N\times1$ matrix which is the first column of
$Y,$ and $Z$ is the $N\times(N-1)$ matrix of the other columns. Then, with
$v_{0}=(1,0,\ldots,0)$, $Yv_{0}=w$, and
\begin{equation}
(1/N)Y^{\prime}SYv_{0}=(1/N)\left(
\begin{array}
[c]{c}%
w^{\prime}\\
Z^{\prime}%
\end{array}
\right)  Sw=(1/N)\left(
\begin{array}
[c]{c}%
w^{\prime}Sw\\
Z^{\prime}Sw
\end{array}
\right)  .
\end{equation}
Note that the random $Z$ and $w$ are \emph{independent}. $Z$ is an
$N\times(N-1)$ matrix with entries i.i.d.\ $\sim{\mathcal{N}}$, and by the
symmetries of this distribution\thinspace(as in Appendices \ref{sbsc:Sum} and
\ref{sbsc:Mat}), $(1/N)Z^{\prime}Sw$ is distributed like $(1/N)\Vert Sw\Vert$
times an $(N-1)$ vector with entries i.i.d.\ $\sim{\mathcal{N}}$ -- near
$(1/\sqrt{N})\Vert Sw\Vert$ times a vector uniformly distributed on $S^{N-2}$.
And, as in Appendix \ref{sbsc:LinOp}, $(1/\sqrt{N})\Vert Sw\Vert$ is
concentrated near $\Vert(s_{1},s_{2},\ldots,s_{N})\Vert_{2}^{(\pi)}$. As for
$(1/N)w^{\prime}Sw$ -- it is just (\ref{eq:S}) -- its value is concentrated
near $(1/N)\operatorname*{tr}\,A$, which if $A$ is positive-semidefinite is
equal to $\Vert(s_{1},s_{2},\ldots,s_{N})\Vert_{1}^{(\pi)}$.

To conclude, the value our random $A$ gives to $(1,0,\ldots,0)$ is a vector
with first coordinate near $(1/N)\operatorname*{tr}\,A$ -- which if $A$ is
positive-semidefinite is $\Vert(s_{1},s_{2},\ldots,s_{N})\Vert_{1}^{(\pi)}$,
and other coordinates forming a vector near the product of $\Vert(s_{1}%
,s_{2},\ldots,s_{N})\Vert_{2}^{(\pi)}$ with a vector uniformly distributed on
$S^{N-2}$. Its norm is $\Vert(s_{1},s_{2},\ldots,s_{N})\Vert_{2}^{(\pi)}$ up
to a deviation $O(1/N)$, and one obtains values agreeing with the above for
$\langle Ax,x\rangle$ and the square of the distance between $v$ and $Av$,
both normalized.

In particular, for $A$ symmetric, employing uniform distribution on spheres in
the Markov chain as in Appendix \ref{sbsc:Sphere} and Conclusion
\ref{conc:SphereVectors} is vindicated. Therefore, for a product of a sequence
of symmetric operators $A_{M}A_{M-1}\cdots A_{1}$, we may apply Conclusion
\ref{conc:SphereVectors} to obtain the following conclusion.

\begin{conclusion}
\label{conc:ProdMat} For a product $A_{M}A_{M-1}\cdots A_{1}$, of a sequence
of \emph{symmetric} operators $A_{i}$ with given eigenvalues $(s_{1}%
^{(i)},s_{2}^{(i)},\ldots,s_{N}^{(i)})$, the square of the distance between
$v$ and $A_{M}A_{M-1}\cdots A_{1}v$, both normalized to norm $1$, is, with
almost full probability, near (with deviation $O(\sqrt{M}/\sqrt{N})$)%
\begin{equation}
2\left(  1-\prod_{i=1}^{M}\dfrac{(1/N)\operatorname*{tr}\,A_{i}}{(1/\sqrt
{N})\Vert A_{i}\Vert_{HS}}\right)  =2\left(  1-\prod_{i=1}^{M}\dfrac
{(1/N)\operatorname*{tr}A_{i}}{\Vert(s_{1}^{(i)},s_{2}^{(i)},\ldots
,s_{N}^{(i)})\Vert_{2}^{(\pi)}}\right)  ,
\end{equation}
which, if for all $i,$ $A_{i}$ is positive semidefinite, is equal to%
\begin{equation}
2\left(  1-\prod_{i=1}^{M}\dfrac{\Vert(s_{1}^{(i)},s_{2}^{(i)},\ldots
,s_{N}^{(i)})\Vert_{1}^{(\pi)}}{\Vert(s_{1}^{(i)},s_{2}^{(i)},\ldots
,s_{N}^{(i)})\Vert_{2}^{(\pi)}}\right)  .\label{eq:ProdMat}%
\end{equation}

\end{conclusion}

\begin{remark}
\label{remark:ProdMat} Note that \emph{if the $A_{i}$ are }positive
semidefinite, the value (\ref{eq:ProdMat}) around which the square of the
distance between the points on $S^{N-1}$ is concentrated, is $\leq2$, that is,
the distance is $\leq\sqrt{2}$ and the angle between the vectors is
$\leq90^{\circ}$.\bigskip
\end{remark}

\textbf{Acknowledgments.} We thank two anonymous reviewers for their
constructive comments. This work was supported by research grant no. 2013003
of the United States-Israel Binational Science Foundation (BSF).\bigskip

\textbf{Conflict of Interest}. The authors declare that they have no conflict
of interest.

\end{document}